\documentclass{gtart}
\usepackage{amsmath, amssymb,  graphicx, epsfig, verbatim,  psfrag, color}
\usepackage[small,nohug,heads=littlevee]{diagrams} 

\newtheorem{thm}{Theorem}[section]

\newtheorem{cor}[thm]{Corollary}
\newtheorem{lemma}[thm]{Lemma}
\newtheorem{prop}[thm]{Proposition}
\newtheorem{defn}[thm]{Definition}

\newtheorem{remark}[thm]{Remark}
\newtheorem{rems}[thm]{Remarks}
\numberwithin{equation}{section}

\newcommand{\sign}{S}
\newcommand{\kaa}{\mathbf k}
\newcommand{\laa}{\mathbf l}

\newcounter{bean}

\newcommand\Field{\mathbb F}

\newcommand\Dual{\mathcal D}
\newcommand\Duality\Dual

\newcommand\x{\mathbf x}
\newcommand\w{\mathbf w}
\newcommand\z{\mathbf z}

\newcommand\y{\mathbf y}

\newcommand\ModSphere{\ModFlow\left({\mathbb S}\longrightarrow 
\Sym^{g-1}(\Sigma_{1})\times \Sym^2(\Sigma_{2})\right)}
\newcommand\ModSpheres\ModSphere

\newcommand\UnparModSp{\widehat \ModSp}
\newcommand\UnparModFlow\UnparModSp
\newcommand\Mod\ModSp

\newcommand\ModMaps{\mathcal M}
\newcommand\ModSp\ModMaps

\newcommand\Ta{{\mathbb T}_{\alpha}}
\newcommand\Tb{{\mathbb T}_{\beta}}

\newcommand\alphas{\mbox{\boldmath$\alpha$}}

\newcommand\betas{\mbox{\boldmath$\beta$}}

\newcommand{\bfz}{{\mathbb {Z}}}

\newcommand{\bfq}{{\mathbb {Q}}}

\newcommand{\Flows}{\mbox{Flows}}

\newcommand{\DD}{\mathfrak D}

 \newcommand{\Z}{\mathbb Z}   

\newcommand{\partiala}{\widehat{\partial}}

\newcommand{\HFa}{{\widehat {\rm {HF}}}}
\newcommand{\CFa}{{\widehat {\rm {CF}}}}
\newcommand{\CFaa}{{\widetilde {\rm {CF}}}}
\newcommand{\HFaa}{{\widetilde {\rm {HF}}}}
\newcommand{\HFast}{{\widehat {\rm {HF}}}_{{\rm {st}}}}
\newcommand{\partialaa}{{\widetilde\partial}}

\newcommand{\alphak}{\alphas}
\newcommand{\betak}{\betas}

\newcommand\gauge{u}

\begin{document}

\title{Combinatorial Heegaard Floer homology and sign assignments}

\author{Peter Ozsv\'ath}
\address{Department of Mathematics, Princeton University,\\ 
Princeton, NJ, 08544}
\email{petero@math.princeton.edu}

\author{Andr\'{a}s I. Stipsicz}
\address{R{\'e}nyi Institute of Mathematics\\
Budapest, Hungary, and\\
Institute for Advanced Study, Princeton, NJ}
\email{stipsicz@math-inst.hu}

\author{Zolt\'an Szab\'o}
\address{Department of Mathematics, Princeton University\\
Princeton, NJ, 08544}
\email{szabo@math.princeton.edu}



\begin{abstract}
We provide an intergral lift of the combinatorial definition of
Heegaard Floer homology for nice diagrams, and show that the
proof of independence using convenient diagrams adapts to this setting.
\end{abstract}

\maketitle

\section{Introduction}
\label{sec:first}
In \cite{OSzF1, OSzF2} various versions of Heegaard Floer homology
groups were defined for oriented, closed 3--manifolds. The
construction of these invariants relied on a Heegaard diagram of the
3--manifold, and applied a suitably adapted variant of Lagrangian
Floer homology to a symplectic manifold associated to the Heegaard
diagram. Consequently, both the definition of the group and the
verification of its topological invariance involved (almost--)complex
analytic arguments. The homology groups come with additional
structures, such as a decomposition according to spin$^c$ structures
of the 3--manifold, and an absolute $\bfq$--grading of those groups
which correspond to torsion spin$^c$ structures. The theories can be
most easily defined over the base field $\Field =\Z /2\Z$; but, with
the help of coherent orientation systems, they also admit a definition
over $\Z$. For the purposes of three-dimensional topological
applications, the theory over $\Field$ is often sufficient. For
four-dimensional applications, most notably those using the mixed
invariant defined in \cite{OSzfour}, however, the integral valued
theory is much more powerful than its mod $2$ reduced counterpart.

In \cite{SW} Sarkar and Wang found a combinatorial way for computing
the simplest version, $\HFa $ of the Heegaard Floer homology groups
over the field $\Field = \bfz /2\bfz$.  Their idea was to use Heegaard
diagrams with a particular combinatorial structure, called {\em{nice
    diagrams}}, in the computation of the Heegaard Floer homology of a
given 3--manifold $Y$. For such diagrams the (almost--)holomorphic
computations reduce to simple combinatorics. They also showed that any
3--manifold admits a nice diagram.  In \cite{nice}, the present
authors described certain specific nice diagrams. With the help of
these \emph{convenient} diagrams, we verified the independence of the
(stable) Heegaard Floer homology groups from the choice of the
diagram, using a purely topological argument. In \cite{nice, SW} only
$\Z /2\Z$--coefficients were used.

In the present work we extend the combinatorial/topological approach
from \cite{nice}, to provide a combinatorial definition
of the (stable) $\HFa$--version over $\Z$. To state our main results,
we first recall the basics of the definition of Heegaard Floer
homology groups.

Suppose that $\DD =(\Sigma , \alphak , \betak , \w )$ is a given nice
multi-pointed Heegaard diagram of a 3--manifold $Y$. The Heegaard
Floer chain complex $( \CFaa (\DD ), \partialaa _{\DD})$ of $\DD$ over
$\Z/2\Z$ is defined by considering the $\Z/2\Z$--vector space
generated by the \emph{generators}, i.e. $n$--tuples $\x=\{ x_1,
\ldots , x_n\} \subset \Sigma$ with the property that each $\alpha _i
\in \alphak$ and each $\beta _j \in \betak$ contains exactly one
element of $\x$.  The boundary operator $\partialaa _{\DD}$ is the
linear map $\CFaa (\DD )\to \CFaa (\DD )$ given by the matrix element
$\langle \partialaa _{\DD}\x , \y \rangle$, which is equal to the mod
2 count of \emph{flows} (i.e., empty bigons or empty rectangles) from
$\x $ to $\y$. (For a more detailed treatment see
\cite[Section~6]{nice}.)  The map $\partialaa _{\DD}$ then satisfies
$\partialaa _{\DD} ^2=0$.  In \cite{nice} we showed that the homology
of the resulting chain complex is (stably) invariant under \emph{nice
  moves}. In \cite{nice} it was also shown that specific nice diagrams
(called \emph{convenient}) can be connected by sequences of nice
moves.  This result then completed the proof of the topological
invariance of the stable groups. (For a more detailed description of
these notions see \cite{nice}.)

The definition above can be adapted to the setting over $\Z$: now
$\CFaa ( \DD ; \Z )$ is generated by the same set of generators over
$\Z$, while the matrix element $\langle \partialaa _{\DD}\x , \y
\rangle$ of the boundary map $\partialaa _{\DD}^{\Z}$ counts the empty
bigons and empty rectangles with certain sign. The aim of this paper
is to describe a sign assignment for these objects which has two
crucial properties: (1) the resulting operator $\partialaa ^{\Z
}_{\DD}$ satisfies $(\partialaa ^{\Z }_{\DD})^2=0$, and (2) the
homology of the resulting chain complex is (stably) invariant under
nice moves.

Our strategy is as follows: we first define \emph{formal generators}
and \emph{formal flows}, which capture certain combinatorial features
of actual generators and flows associated to a Heegaard diagram.  (We
use the term {\em flow} loosely to refer to those objects which are
counted in the Heegaard Floer differential.) For a fixed positive
integer $n$, in Section~\ref{sec:second} we will define the set
${\mathcal {G}}_n$ of \emph{formal generators}.  The set ${\mathcal
  {F}}_n$ (also to be defined in Section~\ref{sec:second}) will
consist of \emph{formal flows} connecting pairs of formal generators
from ${\mathcal {G}}_n$.  After fixing some extra data, such as
orientations on the $\alphak$-- and $\betak$--curves and an order on
them, an intersection point (and a flow) in a Heegaard diagram having
$n$ $\alphak$-- and $n$ $\betak$--curves naturally gives rise to a
formal generator in ${\mathcal {G}}_n$ (and a formal flow in ${\mathcal
  {F}}_n$, respectively).

A \emph{sign assignment} is a map $S\colon {\mathcal {F}}_n\to \{ \pm
1\}$ which satisfies certain properties (to be spelled out in
Definition~\ref{def:SignAssignment}).  By applying simple
modifications (see Definition~\ref{def:gauge}) to a given sign
assignment, we can produce further sign assignments, which will be
called \emph{gauge equivalent} to the original sign assignment.  The
main result of the present paper is
\begin{thm}\label{thm:main1}
For a given positive integer $n$ there exists a sign assignment
$S\colon {\mathcal {F}}_n\to \{ \pm 1 \}$, and it is unique up to
gauge equivalence.
\end{thm}
Resting on this result, for a nice Heegaard diagram $\DD$ and a
sign assignment $S$ we will show:
\begin{thm}\label{thm:main2}
  The map $\partialaa _{\DD} ^{\Z}$ over $\Z$ defined using the fixed
  sign assignment $S$ satisfies $(\partialaa _{\DD} ^{\Z})^2=0$ and
  the resulting homology $\HFaa (\DD ; \Z)$ is independent of the
  choice of $S$, of the chosen orientations and order of the
  $\alphak$-- and $\betak$--curves. Moreover $\HFaa (\DD ; \Z)$ is
  (stably) invariant under nice moves.
\end{thm}

As an application of the theorem above, we will show that the stable
Heegaard Floer homology $\HFast (Y)$ of a 3--manifold $Y$ (as it is
defined in \cite{nice}) admits an integral lift over $\Z$ and provides
a diffeomorphism invariant $\HFast (Y; \Z )$ of closed, oriented
3--manifolds; cf. Corollary~\ref{cor:indep}.

The paper is organized as follows. In Section~\ref{sec:second} we
introduce the necessary formal notions and define sign assignments.
To give a better picture about these objects, we also work out two
examples (for powers $n=1,2$). In Section~\ref{sec:third} we apply the
existence and uniqueness result of sign assignments and verify the
independence of the Heegaard Floer homology groups over $\Z$ from the
necessary choices, leading us to the proof of
Theorem~\ref{thm:main2}. Finally, in Section~\ref{sec:fourth} we prove
Theorem~\ref{thm:main1}.

\bigskip

{\bf {Acknowledgements:}} PSO was supported by NSF grant number
DMS-0804121.  AS was supported by OTKA NK81203 and by \emph{Lend\"ulet
  project} of the Hungarian Academy of Sciences. He also wants to
thank the Institute for Advanced Study for their hospitality. ZSz was
supported by NSF grant numbers DMS-0704053 and DMS-1006006.
 We would like to thank the
Mathematical Sciences Research Institute, Berkeley for providing a
productive research enviroment.

\section{Sign assignments}
\label{sec:second}
For the definition of the formal generators and formal flows fix a
positive integer $n$ and two $n$-element sets $\alphas$ and $\betas$.
(In the following discussion $n$ will be referred to as the \emph{power}
of the formal generators, flows and the sign assignments we will
define with the use of the sets $\alphak$ and $\betak$.)

\begin{defn}
  A {\bf formal generator} is a one-to-one correspondence
  $\rho$ between the sets $\alphas$ and $\betas$
  (which we think of as a subset of the Cartesian
  product $\alphas\times\betas$),
  together with a function $\zeta$ from $\rho$ to $\{\pm 1\}$.
\end{defn}
More concretely, after fixing orderings of the elements of $\alphas$
and $\betas$, $\alphas=\{\alpha_1,\dots,\alpha_n\}$ and
$\betas=\{\beta_1,\dots,\beta_n\}$, the one-to-one correspondence
$\rho$ can be thought of as a permutation $\sigma$ of $\{1,\dots,n\}$,
via the convention that
$\rho=\{(\alpha_i,\beta_{\sigma(i)})\}_{i=1}^n\subset \alphas \times
\betas$.  Similarly, the function $\zeta$ can be encoded in an
$n$-tuple $\epsilon=(\epsilon_1,\dots,\epsilon_n)\in\{\pm 1\}^n$, with
the convention that $\epsilon_i=\zeta(\alpha_i,\beta_{\sigma(i)})$.
We call $\sigma$ the {\em associated permutation}, and
$\epsilon\in\{\pm 1\}^n$ the {\em sign profile} of the formal
generator.  With respect to this choice we write formal generators as
pairs $(\epsilon,\sigma)$. For the fixed integer $n$ the set of formal
generators of power $n$ will be denoted by ${\mathcal {G}}_n$. It
follows from the above definition that ${\mathcal {G}}_n$ has $n!\cdot 2^n$
elements.

A pictorial way of defining formal generators is given by considering
$n$ disjoint crosses on the plane, where at each of the crossing
points the two arcs are equipped with an orientation and decorated
with one of the $\alpha _i$ or $\beta _j$. (Each arc is decorated by a
different $\alpha_i$ or $\beta_j$.) We consider two such pictures
identical if there is an orientation preserving self-diffeomorphism of
the plane mapping one picture to the other, respecting both the
labelings and the orientations of the arcs in the crosses. The sign of
the crossing (with the convention that the $\alphak$--curve comes
first and the plane is oriented by the counterclockwise rotation)
determines the sign profile. For a particular example see
Figure~\ref{f:generators}. It is rather simple to derive the abstract
description of a formal generator from its pictorial presentation.
\begin{figure}
  \begin{center}
    \epsfig{file=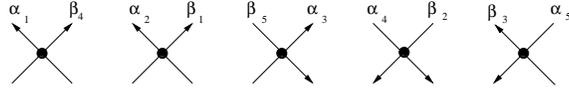}
  \end{center}
  \caption {{\bf A pictorial presentation of a formal generator.}  We
    depict a generator for $n=5$. This generator corresponds to the
    permutation $(142)(35)$ and its sign profile is the constant $-1$
    function.}
    \label{f:generators}
\end{figure}

Now we turn to the definition of formal flows. This will be done in
two steps: we will first define \emph{formal bigons} and then
\emph{formal rectangles}.
 
\begin{defn}\label{def:formalbig}
  For a fixed positive integer $n$ and sets $\alphak =\{ \alpha _1,
  \ldots , \alpha _n\}$, $\betak = \{ \beta _1, \ldots , \beta _n\}$
  consider $n-1$ pairs of oriented arcs in the plane intersecting each
  other in each pair exactly once, and otherwise disjoint. Consider a
  further pair of oriented arcs, intersecting each other in two
  points, and disjoint from all the crossings. The complement of the
  last two arcs has two components (one compact and one non-compact);
  the first $n-1$ pairs are all required to be in the non-compact
  component.  Decorate one of the arcs in each pair with an $\alpha
  _i$ and the other one with a $\beta _j$ in such a manner that each
  element of $\alphak$ and of $\betak$ is used exactly once. Two such
  configurations are considered to be equivalent if there is an
  orientation preserving diffeomorpism of the plane mapping one into
  the other, while respecting both the orientations and the
  decorations of the arcs. An equivalence class of such objects is
  called a {\bf formal bigon}. For a pictorial presentation of a
  formal bigon, see Figure~\ref{f:exabig}.
\end{defn}

\begin{figure}
  \begin{center}
    \includegraphics{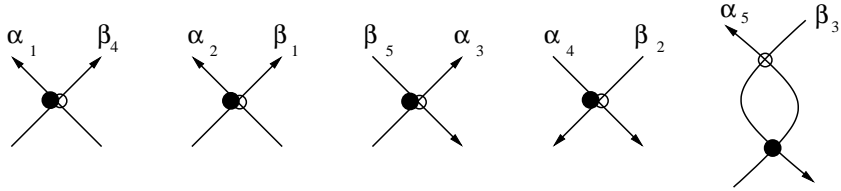}
  \end{center}
  \caption {{\bf A formal bigon.} The bigon in the diagram connects
    two formal generators for $n=5$. The two formal generators
    connected by the formal bigon are represented by the full and
    hollow circles, respectively. In the above example the formal
    bigon points from the generator represented by full circles to
    the one represented by hollow circles.}
    \label{f:exabig}
\end{figure}

A formal bigon determines two formal generators $\x$ and $\y$ by
adding the small neighbourhood of one of the crossings of the last two
arcs (intersecting each other twice) to the first $n-1$ crosses, with
the induced orientations and decorations. The formal bigon is from
$\x$ to $\y$ (denoted by $b\colon \x \to \y$) if the orientation of
the plane restricted to the compact domain encircled by the last two
arcs induces an orientation on the arc with the $\alphak$-decoration
pointing from the $\x$-coordinate towards the $\y$--coordinate. The
four possible formal bigons for $n=1$ are illustrated in
Figure~\ref{fig:OrientedBigons}.
\begin{figure}
  \begin{center}
    \includegraphics[width=3cm]{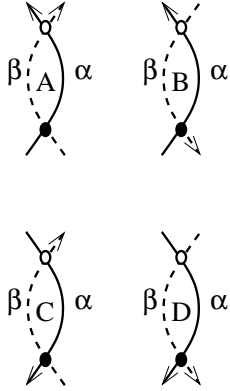}
  \end{center}
  \caption {{\bf The four formal bigons for $n=1$.} Each bigon points
    from the formal generator denoted by the full circle to the one
    denoted by the hollow circle. In this diagram the $\beta$-arcs are
    distinguished from the $\alpha$-arcs by being drwan as dashed
    curves.}
\label{fig:OrientedBigons}
\end{figure}

Notice that the two formal generators $\x$ and $\y$ connected by
a formal bigon $b$ have identical associated permutations, 
while the sign profiles of $\x$ and $\y$ differ in exactly one
coordinate (given by the labels of the $\alphak$ and $\betak$ curve
corresponding to the arcs intersecting each other twice). 
We say that the bigon $b$  is {\em supported in}
this coordinate, or that is its \emph{moving coordinate}. 
For a given $n$ there are $2n\cdot n!\cdot 2^n$ formal bigons: there are 
$n!\cdot 2^n$ choices for the starting formal generator, $n$ choices for the
moving coordinates and 2 further possibilities as how the bigon
starts at the selected crossing containing the moving coordinates.
We make the following analogous definitions for rectangles:

\begin{defn}
\label{def:formalectangle}
For a fixed positive integer $n$ and sets $\alphak =\{ \alpha _1,
\ldots , \alpha _n\}$, $\betak = \{ \beta _1, \ldots , \beta _n\}$
consider $n-2$ pairs of oriented arcs in the plane intersecting each
other in each pair exactly once, and otherwise disjoint. Consider
furthermore two pairs of oriented closed arcs $(a_1, b_1)$ and $(a_2,
b_2)$ such that $a_1$ and $a_2$ (and likewise $b_1$ and $b_2$) are
disjoint, while both $a_i$ intersects both $b_j$ exactly once in their
interiors. One of the two components of the complement is compact, and
we require its interior to be disjoint from all the other arcs.
Decorate one of the arcs in each pair with an $\alpha _i$ and the
other one with a $\beta _j$ in such a manner that each element of
$\alphak$ and of $\betak$ is used exactly once; the $a_i$ arcs in the
rectangle will be decorated by elements of $\alphak$ while the $b_j$
arcs with elements of $\betak$. Two such configurations are considered
to be equivalent if there is an orientation preserving diffeomorpism
of the plane mapping one into the other, while respecting both the
orientations and the decorations of the arcs. An equivalence class of
such objects is called a {\bf formal rectangle}. For a pictorial
presentation of a formal rectangle see Figure~\ref{f:exarect}.
\end{defn}

\begin{figure}
  \begin{center}
    \includegraphics{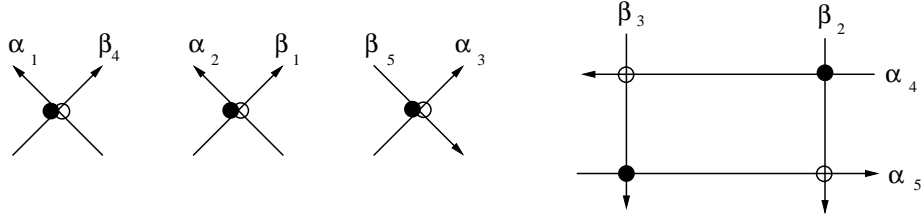}
  \end{center}
  \caption {{\bf A formal rectangle.} Once again, the rectangle given
    by the diagram is a formal flow for $n=5$. The rectangle points
    from the full circles to the hollow circles.}
    \label{f:exarect}
\end{figure}

Notice that the last two pairs of arcs (provided that are made of
straight line segments) form a rectangle with four vertices. The
formal rectangle determines two formal generators $\x$ and $\y$, where
the first $n-2$ coordinates coming from the crosses are completed by
the neighbourhoods of two opposite vertices of the above rectangle.
Once again, using the restriction of the orientation of the plane, we
say that the formal rectangle $r$ is from $\x$ to $\y$ (and write
$r\colon \x \to \y$) if the induced orientation on the sides of the
rectangle labeled by $\alphak$ (viewed as part of the boundary of the
compact component of the complement) point from the $\x$-coordinate to
the $\y$-coordinate. Notice also that the associated permutations for
$\x$ and $\y$ differ by a transposition, and the coordinates in the
transposition are the moving coordinates of the rectangle.  It is easy
to determine the number of formal rectangles when $\vert \alphak \vert
= \vert \betak \vert =n$: there are $n!\cdot 2^n$ starting points of a
rectangle, and once this is fixed, we have $\frac{1}{2}n(n-1)$
possibilities to choose the moving coordinates.  In addition, there
are 2 ways at each of the two starting coordinates the rectangle can
start. Altogether it shows that there are $2n\cdot (n-1)\cdot n!\cdot
2^n$ formal rectangles of power $n$.

\begin{defn}\label{def:fflow}
A {\bf formal flow} is, by definition, either a formal bigon or a
formal rectangle. For a given positive integer $n$ the set of formal
flows connecting elements of the set ${\mathcal {G}}_n$ of formal
generators will be denoted by ${\mathcal {F}}_n$.
\end{defn}

The sign assignment we are looking for is a map from ${\mathcal
  {F}}_n$ to $\{ \pm 1\}$ which satisfies certain relations, which we
describe now. Consider one of the diagrams of Figure~\ref{f:degen}.
\begin{figure}
  \begin{center}
    \includegraphics{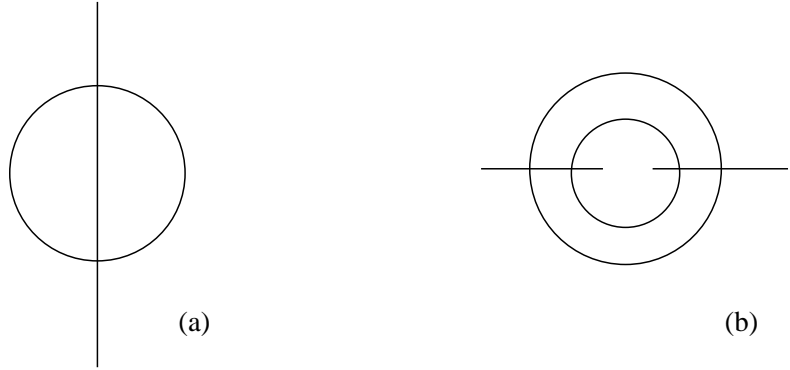}
  \end{center}
  \caption {{\bf Two pairs of formal flows giving rise to boundary
      degenerations.} The diagram on the left shows a  disk-like, 
while on  the right an annular boundary degeneration. In both cases we can 
equip the curves with arbitrary orientations, and decorations from 
the sets $\alphak$ and $\betak$. As always, curves with the same type
of decorations should be disjoint. }
    \label{f:degen}
\end{figure}
Suppose that, with some orientations and after decorating the arcs
with $\alpha _i$ and $\beta _j$ (and adding the oriented, decorated
crossings), Figure~\ref{f:degen}(a) or (b) represent two formal flows $\phi
_1$ and $\phi _2$.  Then we say that the pair $(\phi _1, \phi _2)$ is
a \emph{boundary degeneration}.  The type of the degeneration is
$\alpha$ or $\beta$, depending on the decoration of the circle(s) in the
figure. Sometimes we say that in case (a) the degeneration is
disk-like, while in (b) it is annular. Notice that if $\phi_1$ and
$\phi_2$ are two formal flows which give a pair of boundary
degeneration, and $\phi _1$ is a formal flow from one formal generator
$(\epsilon _1 ,\sigma _1)$ to another one $(\epsilon _2, \sigma _2)$
then $\phi _2$ is a formal flow from $(\epsilon _2, \sigma _2)$ back
to $(\epsilon _1, \sigma _1)$.

Similarly, consider a pair of formal flows $(\phi _1, \phi _2)$ with
the property that $\phi _1$ goes from $(\epsilon _1, \sigma _1)$ to
$(\epsilon _2, \sigma _2)$, while $\phi _2$ goes from $( \epsilon _2,
\sigma_2)$ to $(\epsilon _3, \sigma _3)$, and now assume that
$(\epsilon _1, \sigma _1)$ is different from $(\epsilon_3, \sigma
_3)$. We distinguish two cases. First, if the coordinates which move
under $\phi _1$ are different from the ones moving under $\phi _2$,
then we can switch the order of these flows to provide two further
flows $\phi _3\colon (\epsilon _1, \sigma _1)\to (\epsilon _2', \sigma
_2')$ and $\phi _4\colon (\epsilon _2', \sigma _2')\to (\epsilon _3,
\sigma _3)$: $\phi_3$ is uniquely determined by the properties that it
has the same initial point as $\phi_1$ but the moving coordinates of
$\phi_2$, whereas $\phi_4$ has the same terminal point as $\phi_2$ but
the same moving coordinates as $\phi_1$.  In this case, we say that
the two pairs $(\phi _1, \phi _2)$ and $(\phi _3, \phi _4)$ form a
\emph{square}. In case there are moving coordinates shared by $\phi
_1$ and $\phi _2$, we consider one of the diagrams of
Figure~\ref{f:squares} (equipped with all possible $\alpha
_i$- and $\beta _j$-curves and orientations, and extended by all
possible oriented crossings), which define the corresponding pair of
formal flows $(\phi _3, \phi _4)$.  Once again, in such a situation we
say that the pairs $(\phi _1, \phi _2)$ and $(\phi _3, \phi _4)$ form
a \emph{square}.  Now we are in the position of giving the definition
of a sign assignment.

\begin{figure}
  \begin{center}
    \includegraphics{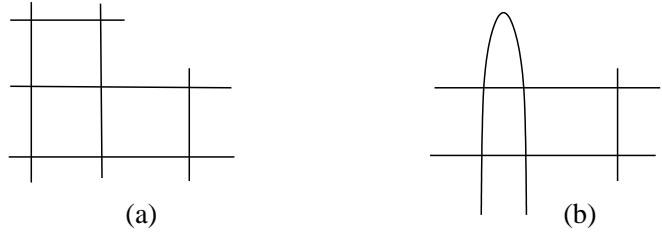}
  \end{center}
  \caption {{\bf Diagrams describing flows sharing moving coordinates.}
The curve segments can be equipped with arbitrary orientations, labelings and
extensions with further crossings to define squares of flow lines
$\{ (\phi _1, \phi _2), (\phi _3, \phi _4) \}$.}
    \label{f:squares}
\end{figure}

\begin{defn}
  \label{def:SignAssignment}
  Fix a positive integer $n$.  A {\bf sign assignment} $S$ of power
  $n$ is a map from the set of all formal flows ${\mathcal {F}}_n$
  into $\{ \pm 1\}$ with the following properties:
  \begin{list}
    {(S-\arabic{bean})}{\usecounter{bean}\setlength{\rightmargin}{\leftmargin}}
    \item
      \label{property:AlphaDegenerations}
      if $(\phi_1, \phi_2)$ is an $\alpha$-type boundary degeneration, then
      $$\sign(\phi_1)\cdot \sign(\phi_2)=1;$$
    \item if 
      \label{property:BetaDegenerations}
      $(\phi_1, \phi_2)$ is a $\beta$-type boundary degeneration, then
      $$\sign(\phi_1)\cdot \sign(\phi_2)=-1;$$
    \item 
      \label{property:AntiCommutation}
      if the two pairs $(\phi_1, \phi_2)$ and $(\phi_3 ,\phi_4)$ form
      a square, then
      $$\sign(\phi_1)\cdot \sign(\phi_2)+\sign(\phi_3)\cdot \sign(\phi_4)=0.$$
  \end{list}
  Notice that this last requirement is equivalent to requiring the
  identity $\Pi _{i=1}^4S(\phi _i)=-1$ to hold. 
\end{defn}
There is a simple operation for constructing new sign assignments from
an old one.

\begin{defn}\label{def:gauge}
  If $\sign \colon {\mathcal {F}}_n\to \{ \pm 1\}$ is a sign
  assignment, and $\gauge$ is any map $\gauge \colon {\mathcal {G}}_n
  \to \{ \pm 1\}$, then we can define a new sign assignment $\sign ^u$
  as follows: if $\phi\colon \x \to \y$ is a formal flow from $\x$ to
  $\y \in {\mathcal {G}}_n$, then let $\sign ^u(\phi)=\gauge(\x)\cdot
  \sign(\phi)\cdot \gauge (\y )$.  If $\sign$ and $\sign ^u$ are
  related in this way, we say that $\sign$ and $\sign^u$ are {\bf
    gauge equivalent} sign assignments and $u$ will be called a {\bf
    gauge transformation}.  The function $u\colon {\mathcal {G}}_n \to
  \{ \pm 1\}$ is a {\bf restricted} gauge transformation if $u(\x )$
  depends only on the permutation corresponding to the formal
  generator $\x$ (and is independent of its sign profile).
\end{defn}
Since in each relation of Definition~\ref{def:SignAssignment} for
$S^u$ any $u (\x )$ appears an even number of times, the fact that
$S^u$ is a sign assignment follows trivially from the fact that $S$ is
a sign assignment. With these definitions in place, we have the
precise version of Theorem~\ref{thm:main1}:
\begin{thm}\label{thm:exunique}
  For any integer $n$ there is, up to gauge equivalence, a unique sign
  assignment on ${\mathcal {F}}_n$.
\end{thm}

\begin{remark}
The definition of a sign assignment shows a certain asymmetry between
the $\alphak$ and $\betak$ curves in the degeneration rule. Let $m\colon
{\mathcal {G}}_n\to \{ \pm 1 \}$ denote the map which associates to
each formal generator $(\sigma , \epsilon )$ the product $sgn(\sigma)
\cdot \Pi \epsilon _i$, where $sgn(\sigma)$ is the parity of the permutation
(and is $1$ for even and $-1$ for odd permutations) and $\epsilon _i$ are
the coordinates of the sign profile $\epsilon$. Then the 
formula $S'(\phi )=S(\phi )\cdot m(\x)$ for a formal flow
$\phi \in {\mathcal {F}}_n$ from $\x$ to $\y$ and for a signs assignment
$S$ 
defines a map $S'\colon {\mathcal {F}}_n \to \{ \pm 1\}$
which satisfies the axioms of a sign assignment provided the roles
of $\alphak$ and $\betak$ are switched.
\end{remark}

There are a number of further types of squares $\{ (\phi _1, \phi _2),
(\phi _3, \phi_4) \}$ with the property that $\phi _1$ and $\phi _2$
(and so also $\phi _3$ and $\phi _4$) share moving coordinates.  In
the diagrams of Figure~\ref{f:squares} only a few such types are
shown. 
It can be easily verified that if the relations required by
Definition~\ref{def:SignAssignment} are satisfied, then the relations
presented by the further squares of Figure~\ref{fig:ujrajz} follow:

\begin{figure}
  \begin{center}
    \includegraphics[width=6cm]{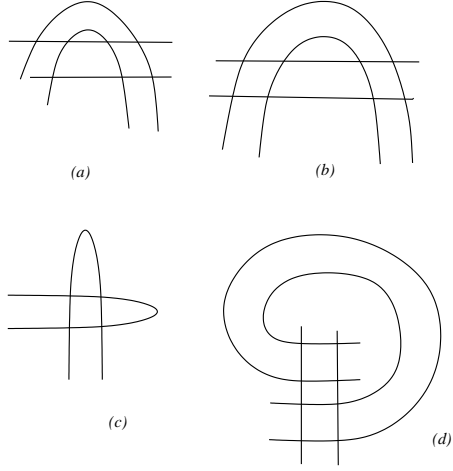}
  \end{center}
  \caption {{\bf Further squares of flows sharing moving coordinates.}}
    \label{fig:ujrajz}
\end{figure}

\begin{lemma}\label{l:tovabbiesetek}
Suppose that the square $\{ (\phi _1, \phi _2), (\phi _3, \phi_4) \}$ 
is defined by one of the diagrams of Figure~\ref{fig:ujrajz}.
For a sign assignement $S$ then we have that
\[
\Pi _{i=1}^4 S(\phi _i)=-1.
\]
\end{lemma}
\begin{proof}
The proof of this statement is a rather long but simple computation. Below we 
show it in one demonstrative case and leave the interested reader to complete
the remaining cases.

Consider the situation depicted by Figure~\ref{fig:ujrajz}(a) and
equip the edges with some orientation and decoration (see, for example,
Figure~\ref{f:utso}(a)).  With the notations of Figure~\ref{f:utso} the
relations of Figure~\ref{f:squares} imply
\begin{align*}
S(XAB)\cdot S(Y)\cdot S(AB)\cdot S(XY)=-1 ,\\
S(XAB)\cdot S(UV)\cdot S(XUAB)\cdot S(V)=-1.
\end{align*}
(Notice that a flow is specified by its initial generator and its
support; above we only indicate the support while the initial
generators follow from the order of the terms.) In addition, $XY$ and
$AD$ differ by a $\beta$ boundary degeneration and the switch of the
sign profile of the non-moving coordinate (which can be realized by
anticommuting with an appropriate bigon), and the same difference
applies to the pair $DC$ and $V$, while $Y$ and $UV$ are identical as
formal flows. Putting all these together, and using the identity of
(S-\ref{property:AntiCommutation}) for the squares of
Figure~\ref{f:utso}(b) and (c), the identity
\begin{align*}
S(AB)\cdot S(CD)\cdot S(AD)\cdot S(BC)=-1
\end{align*}
follows at once. With the chosen orientation and decoration this is
exactly the relation provided by Figure~\ref{f:utso}(a).
\begin{figure}
  \begin{center}
    \includegraphics[width=8cm]{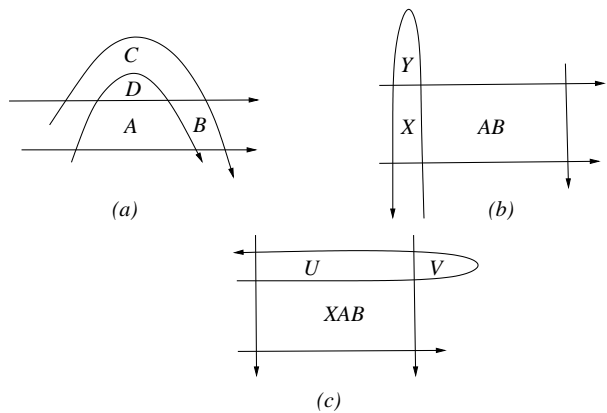}
  \end{center}
  \caption {{\bf The proof of anticommutativity.}}
    \label{f:utso}
\end{figure}

A similar argument verifies the result for the situation depicted by
Figure~\ref{fig:ujrajz}(b).  The identity for pairs shown by
Figures~\ref{fig:ujrajz}(c) and (d) are even simpler: here we only
need to apply boundary degenerations. (Details of these cases are left
to the reader; for Figure~\ref{fig:ujrajz}(c) see also the discussion
prior to Remark~\ref{rmk:BigonOrientations}.)
\end{proof}

The proof of Theorem~\ref{thm:exunique} (given in
Section~\ref{sec:fourth}) is rather long. To give a better picture
about our argument, below we summarize the main steps in the proof. It
starts with the observation that both ${\mathcal {G}}_1$ and
${\mathcal {F}}_1$ are rather simple sets, hence for $n=1$ the
construction (and the proof of uniqueness, up to gauge equivalence) of
a sign assignment is an easy task. Indeed, we will present it in the
subsection below. In the next step, using the $n=1$ case and the usual
principle of signs in singular and simplicial homology, we verify the
statement of Theorem~\ref{thm:exunique} for the subset of ${\mathcal
  {F}}_n$ given by all formal bigons,
cf. Subsection~\ref{subsec:OrientBigons}.  Next we consider another
subset of ${\mathcal {F}}_n$, the flows between formal generators with
sign profile constant {\bf {1}}. These are necessarily formal
rectangles, and these can be modelled in grid diagrams of the
appropriate size. Sign assignments for certain specific rectangles in
grids have been discussed in \cite{MOST}; in
Subsection~\ref{subsec:FixSignProfile} we extend that result to all
the formal rectangles between generators of the fixed sign
profile. 
Finally, in Subsection~\ref{subsec:vary} we use the relations
given by those squares which involve rectangles and bigons to extend
the definition to rectangles with various sign profiles, and we arrive
to the definition of a sign assignment (once a choice of it for bigons
and rectangles among generators of constant {\bf {1}} sign profile is
fixed).  The verification of the properties of a sign assignment
listed in Definition~\ref{def:SignAssignment} will conclude the proof
of Theorem~\ref{thm:exunique}.  We also note that in most of the
proofs very similar statements must be checked for different, but
rather similar objects and configurations. In these cases we typically
pick representative cases, give the argument in detail for those, and
only indicate the necessary modifications for the other cases (in case
there are any significant necessary modifications).

The proof of Theorem~\ref{thm:exunique} will be preceded by its main
application in the proof of Theorem~\ref{thm:main2}.  Before turning
to this application, however, we work out two specific cases of
Theorem~\ref{thm:exunique} for $n=1,2$.

\subsection{Two examples}

\begin{lemma}
  \label{lem:OrientBigons}
  In the case $n=1$ there is a unique sign assignment $\sign_0$, up
  to gauge equivalence.
\end{lemma}
\begin{proof}
Notice that for $n=1$ we only need to deal with formal bigons. We have
two formal generators (differing in their sign profile), and these are
connected by the four formal bigons $A,B,C,D$ shown in
Figure~\ref{fig:OrientedBigons}.

Considering the possible decompositions of an $\alpha$-boundary
degeneration, we conclude that $\sign(A)\cdot \sign(B)=1$ and
$\sign(C)\cdot \sign(D)=1$. (This is gotten by taking an
$\alphak$-circle cut in two along a $\betak$-arc, and considering the
possible orientations of the circle and the arc.)  Similarly, if we
decompose $\beta$-boundary degenerations, we obtain the relations
$\sign(A)\cdot \sign(C)=\sign(B)\cdot \sign(D)=-1$. Putting all these
relations together, we conclude that
  $$\sign(A)=\sign(B)=-\sign(C)=-\sign(D).$$

  There are two possible such sign assignments, which are
  distinguished by their value on $A$; $\sign _0(A)=1$, and
  $\sign _0'(A)=-1$. These two sign assignments are equivalent, using
  the gauge transformation $\gauge(\x)=\epsilon(\x)$.  
\end{proof}
As a further example, we show how the relation given by
Figure~\ref{fig:ujrajz}(a) follows in this simple case: with the
notations of Figure~\ref{f:partic}
\begin{figure}
  \begin{center}
    \includegraphics{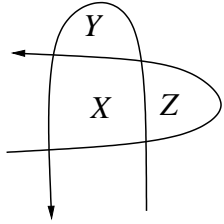}
  \end{center}
  \caption {{\bf A particular case for  $n=1$.}}
    \label{f:partic}
\end{figure}
for the domains $X,Y,Z$ and of Figure~\ref{fig:OrientedBigons} for
$A,B,C,D$, we have $XY=B$, $Y=A$, $XZ=B$ and $Z=D$, hence the identity
$S(XY)\cdot S(Y)\cdot S(XZ)\cdot S(Z)=-1$ follows at once. 

\begin{remark}
  \label{rmk:BigonOrientations}
  The proof of Lemma~\ref{lem:OrientBigons} can be summarized as
  follows: if we fix a sign assignment with $n=1$ on one bigon, the
  signs of the other bigons are fixed by the following two rules: the
  sign of a bigon switches if we reverse the orientation of the
  $\alphak$-arc, and it stays the same if we reverse the orientation
  of the $\betak$-arc. Finally, by passing to equivalent assignments,
  we can arrange for a given bigon to have either sign.
  Compare also~\cite{Seidel}.
\end{remark}

\subsubsection*{The case of power $n=2$}
 
We work out the details of the case where $n=2$, to give an example
where rectangles also appear. In this case there are eight formal
generators, since there are two permutations, and for each permutation
there are four different sign profiles. As we already computed, there
are $32$ bigons. This number can be alternatively deduced as follows:
by fixing the permutation ($2$ possibilities), the moving coordinate
($2$ possibilities), the sign profile at the fixed coordinate ($2$
possibilities), we reduced the count to the $n=1$ case, having $4$
bigons. Notice that by fixing the sign assignment on one of the bigons
in each of these eight groups, the argument given for $n=1$ extends
the function to all formal bigons. By composing two appropriate bigons
with different moving coordinates, however, we get additional
relations. A possible choice of signs for the representatives of each
of the eight groups is shown by Figure~\ref{f:n2big}.  The bigons on
the left correspond to the identity permutation, while on the right to
the single nontrivial permutation $\sigma$. By taking $S$ to be equal
to $1$ on $I_1, I_2, I_3, \sigma _1, \sigma _2, \sigma _3$ and $-1$ on
$I_4$ and $\sigma _4$, the application of the rule formulated in
Remark~\ref{rmk:BigonOrientations} above specifies the value of $S$ on
all formal bigons. Notice that by applying a restricted gauge
transformation $u$ to any sign assignement $S$ on the bigons, the new
sign assignment $S^u$ will be equal to $S$ on the bigons.
\begin{figure}
  \begin{center}
    \includegraphics[width=11cm]{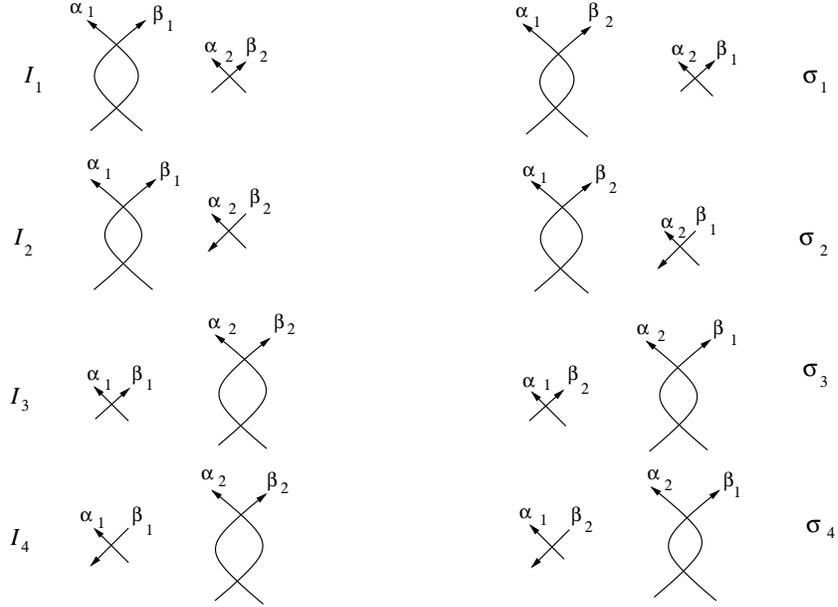}
  \end{center}
  \caption {{\bf Formal bigons for $n=2$.}}
    \label{f:n2big}
\end{figure}

Now we turn to the examination of rectangles. As we determined
earlier, for $n=2$ there are 32 formal rectangles. This can be checked
alternatively as follows: By rotating the rectangle if necessary, we
can assume that at least one of the (vertical) $\betak$--arcs points
up. If both point up, there are two choices as which one is $\beta _1$
and which one is $\beta _2$, and for each such choice there are eight
further choices for the horizontal $\alphak$--curves (orientations and
labels). If only one of the $\betak$--curves points up, then there is
a choice whether it is the left or right, (by rotation we can always
assume that the left one is $\beta _1$), and then we have eight
further choices for the $\alphak$--curves.

Notice that by boundary degenerations we get relations among
rectangles we get by permuting either the $\alphak$-- or the
$\betak$--curves, and we can apply rotations of
$180^{\circ}$. Therefore by fixing the values of $S$ on the eight
rectangles shown by Figure~\ref{f:n2rect}, we have determined the sign
assignment.  Notice that for each $R\neq R_1$, appropriately chosen
bigons, together with $R_1$ and $R$ form a square, hence by fixing
$S(R_1)$ we can determine $S(R)$. (In this step we use the relation
given by the diagram of Figure~\ref{f:squares}(b).)  For example, for
$S(R_1)=1$ a somewhat lenghty but straightforward computation shows
that $S(R_2)=S(R_3)=S(R_8)=1$ and $S(R_4)=S(R_5)=S(R_6)=S(R_7)=-1$.
It remains to check that $S$ is indeed, a sign assignment, which
easily follows since there are no further relations in the
definition. 

Fix the value of the sign assignment $S'$ on $R_1$ to be equal to
$-1$. Consider the restricted gauge transformtation $u$ mapping all
formal generators with associated permutation the identity into $-1$,
and all the others to 1. It is then easy to see that $S^u=S'$.  Notice
that $R_1$ is the single rectangle in this example which connects
formal generators with constant sign profile {\bf {1}}, hence the
above computation demostrates the strategy we described for the proof
of Theorem~\ref{thm:exunique}.

\begin{figure}
  \begin{center}
    \includegraphics[width=10cm]{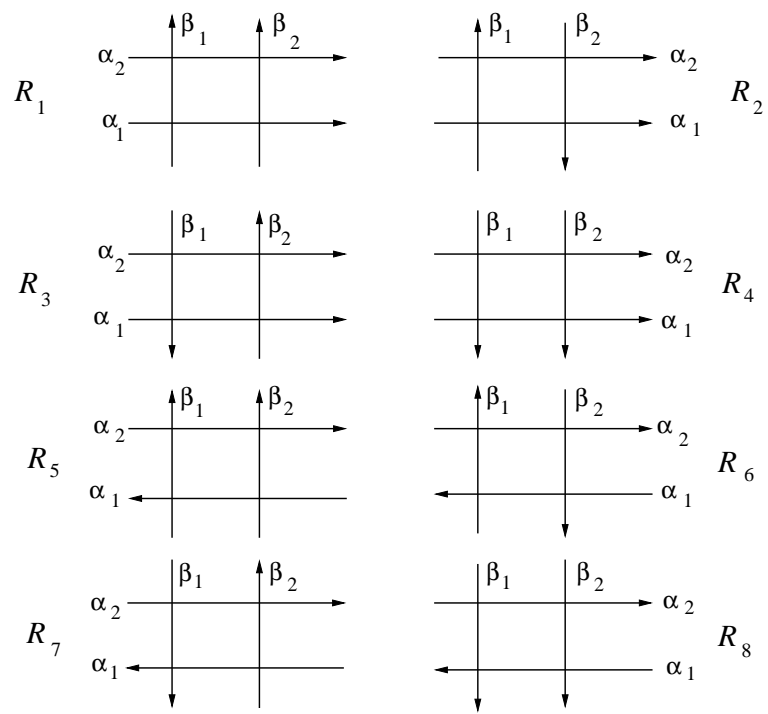}
  \end{center}
  \caption {{\bf Formal rectangles for $n=2$.}}
    \label{f:n2rect}
\end{figure}

\section{Heegaard Floer groups with integer coefficients}
\label{sec:third}

Before we turn to the proof of Theorem~\ref{thm:exunique}, we provide
its main application, namely that nice moves do not change the
(stable) Heegaard Floer homology groups, when defined over $\Z$.
In this section we will heavily rely on notations, definitions,
proofs and results from \cite{nice}.

Suppose that $\DD=(\Sigma , \alphak , \betak , \w )$ is a given nice
Heegaard diagram for a 3--manifold $Y$. Fix an order on the
$\alphak$-- and on the $\betak$--curves, and furthermore orient each
of these curves. Then the generators of the Heegaard Floer chain
complex $\CFaa (\DD ; \Z )$ over $\Z$ naturally define formal
generators of power $\vert \alphak \vert$, while the empty bigons and
empty rectangles (used in the definition of the boundary map) specify
formal flows of the same power.  Fix a sign assignment $S$ of power
$\vert \alphak \vert$ and define the boundary operator $\partialaa
_{\DD} ^{\Z }$ using this sign assignment:
\[
\partialaa _{\DD }^{\Z }(\x) =\sum _{\y \in \Ta \cap \Tb}\sum _{\phi
  \in \Flows (\x , \y )} S(F(\phi ))\cdot \y ,
\]
where $\Flows (\x , \y )\subset \pi _2 (\x, \y)$ denotes the set of
empty bigons or rectangles from $\x$ to $\y$ and $F(\phi )$ is the
formal flow corresponding to $\phi\in \Flows (\x , \y) $.

\begin{thm}\label{thm:squarezero}
The boundary operator $\partialaa _{\DD }^{\Z}$ satisfies 
$(\partialaa _{\DD }^{\Z})^2=0$. 
\end{thm}
\begin{proof}
  In the verification of the mod 2 version of the theorem (presented
  in \cite[Theorem~6.11]{nice}), we show that if $\phi _1 \in \pi _2
  (\x , \y )$ and $\phi _2 \in \pi _2 (\y , \z)$ are empty bigons or
  rectangles, then for the pair $(\phi _1, \phi _2)$ there is another
  pair $(\phi _3, \phi _4)$ such that the two pairs form a
  square. Indeed, if $\phi _1$ and $\phi _2$ have disjoint moving
  coordinates, then $(\phi _3, \phi _4)$ can be given by the flows
  with the same support in the opposite order (in the appropriate
  sense, discussed after Definition~\ref{def:fflow}).  If the two
  flows $\phi _1$ and $\phi _2$ share moving coordinates, then the
  argument given in \cite[Theorem~6.11]{nice} (resting on simple
  planar geometry) produces one of the configurations presented in
  Figure~\ref{f:squares} or of Figure~\ref{fig:ujrajz}.  This shows
  that for every pair $(\phi _1, \phi _2)$ from $\x$ to $\z$ there is
  another pair $(\phi _3, \phi _4)$ such that the pairs form a square.
  By definition (and by Lemma~\ref{l:tovabbiesetek}) a sign assignment
  provides zero contribution on such a pair of pairs, consequently we
  get that the matrix element $\langle (\partialaa _{\DD }^{\Z })^2 \x
  , \z \rangle $ is zero for all $\x$ and $\z$. This shows that the
  square of the boundary operator is zero, concluding the proof.
\end{proof}

\begin{thm}\label{thm:indepp}
The homology of the chain complex $(\CFaa (\DD ; \Z ), \partialaa
_{\DD}^{\Z})$ is independent of the chosen sign assignment $S$, the
order of the curves in $\alphak$ and $\betak$ and the chosen
orientation on them.
\end{thm}
\begin{proof}
Let us fix a Heegaard diagram $\DD$, and fix and order of the
$\alphak$-- and $\betak$--curves, and also orient them.  Suppose that
$S$ and $S'$ are sign assignments of power $n=\vert \alphak \vert$,
and denote the resulting boundary maps by $\partialaa _{\DD}^S$ and
$\partialaa _{\DD}^{S'}$, respectively. According to the uniqueness
part of Theorem~\ref{thm:exunique}, the sign assignments $S$ and $S'$
are gauge equivalent, hence there is a map $\gauge$ on the formal
generators into $\{ \pm 1\}$ with the property that $S'(\phi )=\gauge
(\x _f)\cdot S(\phi )\cdot \gauge (\y _f)$ for a formal flow
connecting the formal generators $\x_f$ and $\y_f$.  (In the proof we
will distinguish the formal generators from the actual generators
coming from $\DD$ by a subscript $f$.) Define the linear map $H\colon
\CFaa (\DD ; \Z )\to \CFaa (\DD ; \Z )$ on the generator $\x$ by $H(\x
)=\gauge (F(\x ))\cdot \x$, where $F(\x )$ denotes the formal
generator corresponding to $\x$. Then $H$ provides an isomorphism
between the chain complexes $(\CFaa (\DD ; \Z ), \partialaa _{\DD}^S)$
and $(\CFaa (\DD ; \Z ), \partialaa _{\DD}^{S'})$, verifying the
isomorphism of the homologies.
 
Assume now that we have a fixed sign assignment $S$ for the diagram
$\DD$, and also fixed the order of the curves, but we fix two
different orientations. For simplicity we can assume that the two
orientations differ only on one curve, say on $\alpha _1$.  This curve
corresponds to the curve $\alpha _{1,f}$ of the set $\alphak$ we use
to define formal generators and formal flows.  Let us denote the first
orientation by $o$, while the second one by $o'$.

Define a map $h\colon {\mathcal {F}}_n \to {\mathcal {F}}_n $ on the set of
formal flows by associating to $\phi \in {\mathcal {F}}_n$ the formal
flow $\phi '$ which is identical to $\phi$ except the orientation on
$\alpha _{1,f}$ is switched to its opposite. It is easy to see that the
composition $S_h=S\circ h$ is also a sign assignment.  By the
definition of $h$, the boundary maps $\partialaa _{\DD }^{S,o}$
(defined using the orientation $o$ and the sign assignment $S$) and
$\partialaa _{\DD}^{S_h, o'}$ coincide, hence provide the same
homologies. On the other hand, by the uniqueness of sign assignments
(up to gauge) we have that $S_h$ and $S$ are gauge equivalent, hence
by the argument given above, the boundary maps $\partialaa
_{\DD}^{S_h,o'}$ and $\partialaa _{\DD}^{S,o'}$ provide isomorphic
chain complexes, concluding the proof of independence from the
orientations.

Finally, suppose that we choose two different ordering among the
$\alphak$-- and $\betak$--curves of $\DD$. Once again, the resulting
permutations provide a map $g\colon {\mathcal {F}}_n \to {\mathcal {F}}_n $
on the set of formal flows, and (as above) the fixed sign assignment
$S$ can be pulled back to give rise to a sign assignment $S_g$, which
is gauge equivalent to $S$. The adaptation of the argument above then
concludes the proof.
\end{proof}

Next we turn to the relation between homologies defined by diagrams
differing by a nice move. Recall that the concept of nice moves was
introduced in \cite[Section~3]{nice}, and these moves on a Heegaard
diagram have the distinctive feature that when applied on a nice
Heegaard diagram, they preserve niceness. In addition, a special set
of nice diagrams (called convenient) has been defined in
\cite[Section~4]{nice}, and it was also shown that any two convenient
diagrams of a given 3-manifold can be connected by a sequence of nice
moves. Recall that there are four types of nice moves: nice
stabilizations (of type-$g$ and type-$b$), nice handle slides and nice
isotopies. (Recall that in a stabilization we increase the number of
curves; in a type-$g$ stabilization the genus of the Heegaard surface
also increases, while in a type-$b$ stabilization the Heegaard surface
stays intact, but the number of basepoints grows.)

\begin{prop}\label{p:nstab}
Suppose that the nice diagrams $\DD$ and $\DD '$ differ by a nice
stabilization. Then the Heegaard Floer homologies with integral
coefficients for $\DD$ and $\DD '$ are stably isomorphic.
\end{prop}
\begin{proof}
Notice that when stabilizing a Heegaard diagram, the cardinality of
the curves changes, hence we need to compare chain complexes using
sign assignments of different power.

Suppose first that the nice stabilization is of type-$g$.  Orient the
two new curves $\alpha _{n+1}$ and $\beta _{n+1}$, and fix a sign
assignment of power $(n+1)$. By restricting this sign assignment to
those formal flows for which the permutation leaves $n+1$ fixed, and
the sign profile is given by the sign of the intersection point
$x_{n+1}=\alpha _{n+1}\cap \beta _{n+1}$, we get a sign assignment of
power $n$, which we can use to define signs before the
stabilization. Then it is easy to see that the isomorphism between the
chain complexes before and after the stabilization found in
\cite[Theorem~7.26]{nice} extends to an isomorphism over $\bfz$,
completing the analysis of this case.

We follow a similar line of argument for type-$b$ stabilization: again,
orient the new curves $\alpha _{n+1}$ and $\beta _{n+1}$ (intersecting
each other in $x_u$ and $x_d$), fix a sign assignment of power $n+1$,
and restrict it to those formal flows where the permutation leaves the
last coordinate unchanged. (There are two such subsets, differing in
the sign profile at the last coordinate.)  By appending either $x_u$
or $x_d$ to the generators of the chain complex associated to the
diagram before the stabilization, we get two isomorphic copies of it
in the new chain complex. The isomorphisms obviously respect the sign
assignments. Notice that although the sign assignments might be
different on the two subsets, nevertheless both are sign assignments
on a copy of ${\mathcal {F}}_n$, hence are gauge equivalent, and in
particular provide isomorphic homologies.  In addition, the map
between these two subcomplexes is zero, since the two bigons
connecting $(\x , x_u)$ and $(\x , x_d)$ come with opposite signs, as
can be verified by applying an $\alpha$- and then a
$\beta$-degeneration.
\end{proof}

Although for nice isotopies and nice handle slides the power of the
necessary sign assignement remains unchanged, the isomorphism of the
homologies is more subtle than in the case of stabilizations. For the
sake of completeness, we first recall the main idea of the proof of
invariance over $\Field=\Z /2\Z$, and then we provide the necessary
refinement for the groups over $\Z$.

Suppose that $\DD$ is the diagram before, while $\DD'$ after the nice
isotopy or nice handle slide. The isomorphism between $\HFa (\DD
)=\HFa (\DD ; \Field )$ and $\HFa (\DD')=\HFa (\DD'; \Field )$ in
\cite[Section~7]{nice} was shown by finding a subcomplex $K$ of $\CFa
(\DD ')$ with the property that (a) $K$ is acyclic and (b) the map $\x
\mapsto \x +K$ for the generators of $\CFa (\DD )$ (which naturally
give rise to generators of $\CFa (\DD ')$ as well) is an isomorphism
between $\CFa (\DD )$ and the quotient complex $\CFa (\DD ')/K$. In
this last step the boundary maps on $\CFa (\DD )$ and on the quotient
$\CFa (\DD ')/K$ were compared.  Indeed, we showed that the matrix
element $\langle \partial (\x +K), \y +K\rangle $ in the quotient
complex is equal to the number of \emph{chains} connecting the
generators $\x$ and $\y$ in the Heegaard diagram $\DD '$. (For the
definition of the concept of chain, see
\cite[Definitions~7.8~and~7.19]{nice}.)

The following simple linear algebraic lemma will show the necessary
statement we need to show for extending the isomorphisms of
\cite[Section~7]{nice} from $\Field$ to $\Z$. In the following
statement we will use the notation of \cite[Section~7]{nice}.  Suppose
therefore that $\sign $ is a given sign assignment for $\DD$.  Since
the Heegaard diagrams $\DD$ and $\DD'$ involve the same number of
curves, $\sign$ also provides a sign assignment for $\DD '$.

\begin{lemma}\label{l:alg}
  Suppose that $C=(D_1, \ldots , D_k)$ is a chain of length $n$ from
  $\x$ to $\y$ in the Heegaard diagram $\DD '$. Suppose that the flow
  $D_i$ connects generators $\kaa _i$ and $\laa _{i+1}$ for
  $i=0,\ldots , k-1$ (with $\kaa_0=\x$ and $\laa _{k}=\y$). Let the
  unique flow (of \cite[Lemmas~7.7,~7.18]{nice}) connecting $\kaa _i$
  and $\laa _i$ be denoted by $E_i$.  Then the signed contribution of
  the chain $C$ in the matrix element $\langle \partial ^{\Z}(\x +K),
  \y +K\rangle $ is equal to
\[
(-1)^{k-1}\Pi _{i=1}^k S(D_i)\Pi _{i=1}^{k-1}S(E_i).
\]
\end{lemma}
\begin{proof}
  Consider the element
\[
v=\x +\sum _{i=1}^{k-1}((-1)^i \Pi _{j\leq i}S(D_j)\Pi _{j\leq i}S(E_j))\cdot \kaa _i .
\]
The contributions of $D_i$ and $E_i$ will cancel in $\partial ^{\Z}$,
hence the sign of $\y$ in $\partial ^{\Z}v$ will be equal to the
coefficient of $\kaa _{k-1}$ in the above sum, multiplied with
$S(D_{k})$, the sign of the flow connecting $\kaa _{k-1}$ and
$\y$. The claim then follows at once.
\end{proof}

In the proof of the next proposition therefore we will relate the
number of empty rectangles/bigons connecting $\x$ and $\y$ in $\DD$
(now equipped with signs provided by a chosen sign assignment) and the
number of chains connecting $\x$ and $\y$ in $\DD'$ (once again, with
signs).  In determining this latter sign, we will appeal to
Lemma~\ref{l:alg}.

\begin{prop}\label{p:niso}
  Assume that $\DD$ and $\DD'$ differ by a nice isotopy.  Then the
  homologies of the corresponding chain complexes (over $\Z$) are
  isomorphic.
\end{prop}
\begin{proof}
  Suppose that $\vert \alphak \vert =n$ and fix a sign assignment $S$
  of power $n$.  According to the result of
  \cite[Proposition~7.14]{nice}, a chain in $\DD '$ connecting the two
  generators $\x$ and $\y$ either consists of a single element $D$
  (which was the domain connecting $\x$ and $\y$ already in $\DD$), or
  it is of length 1.  In the first case the flow connecting $\x$ and
  $\y$ appears in both diagrams, giving rise to the same formal flow,
  and hence getting the same sign by the fixed sign assignement.

Suppose now that the chain is of length one.  This means that there
are two further generators $f_i\kaa $ and $e_i\kaa $ of $\DD '$, and
there is a domain $D_1$ connecting $\x$ to $e_i\kaa$, a domain $D_2$
connecting $f_i\kaa $ to $e_i\kaa$ and finally $D_3$ connecting $f_i
\kaa $ to $\y$.  According to Lemma~\ref{l:alg} (for $k=1$), we need
to show that
\[
S(D)=-S(D_1)\cdot S(D_2)\cdot S(D_3).
\]
The identification of the domains $D_1, D_2, D_3$ based on $D$ and the
nice arc defining the nice isotopy involved two main cases, both
treated in \cite[Proposition~7.14]{nice}. In one case the starting flow $D$
was a rectangle, while in the second it was a bigon. 

Suppose first that $D$ is a rectangle connecting the generators $\x$
and $\y$, and the nice arc $\gamma$ (which defines the nice isotopy)
starts on the side of the rectangle (and then necessarily leaves it,
since $D$ is empty and contains no bigon). As in the proof of
\cite[Proposition~7.14]{nice}, we get the domains $D_1, D_2, D_3$, as
shown on the left of Figure~\ref{f:isotopy}.
\begin{figure}
  \begin{center}
    \includegraphics{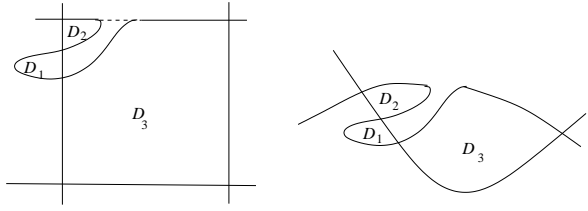}
  \end{center}
  \caption {{\bf Domains in a nice isotopy.}}
    \label{f:isotopy}
\end{figure}
Notice that for an arbitrary choice of orientations of the curves, the
formal flow corresponding to $D$ and to $D_3$ coincide. On the other
hand, it is fairly easy to see that $S(D_1)S(D_2)=-1$, since the two
formal flows can be connected by an $\alpha$-- and a $\beta$--boundary
degeneration, implying the claimed equality.  Essentially the same
argument works in the case $D$ is a bigon, cf. the right diagram of
Figure~\ref{f:isotopy}.  Therefore by Lemma~\ref{l:alg} the map $\CFaa
(\DD ; \Z )\to \CFaa (\DD ' ;\Z )/K$ induced by $\x \mapsto \x +K$
(where the definition of $K$ is lifted from \cite{nice}) gives the
required isomorphism between the homology groups, concluding the
proof.
\end{proof}

\begin{prop}\label{p:nhslide}
  Assume that $\DD$ and $\DD'$ differ by a nice handle slide.  Then
  the homologies of the corresponding chain complexes (over $\Z$) are
  isomorphic.
\end{prop}
\begin{proof}
  Suppose now that $\DD '$ is given by applying a nice handle slide on
  $\DD$. Then, according to \cite[Proposition~7.22]{nice} there are
  chains of length zero, one and two, and these can appear in various
  cases.

Suppose first that the domain connecting $\x$ and $\y$ is a bigon, and the 
nice handle slide applies within one of the elementary rectangles of
the empty bigon. Since the bigon is empty, the handle slide applies
to the boundary arc of the bigon.  The handle slide and the domains are shown
by Figure~\ref{f:slide-big}. 
\begin{figure}
  \begin{center}
    \includegraphics{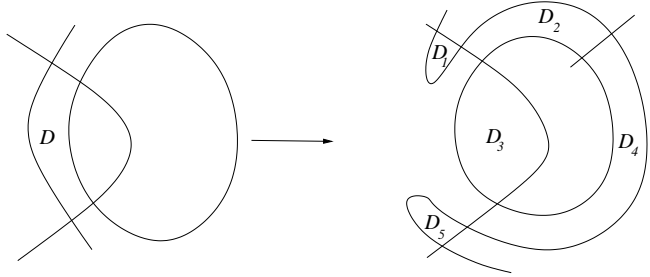}
  \end{center}
  \caption {{\bf Domains in a nice handle slide.} We examine the case when 
$\x$ and $\y$ in $\DD$ are connected by a bigon.}
    \label{f:slide-big}
\end{figure}
Consider now the square given by Figure~\ref{f:slide-square}.
\begin{figure}
  \begin{center}
    \includegraphics{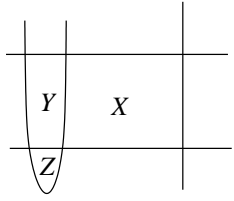}
  \end{center}
  \caption {{\bf The square given by the diagram gives a relation,
      showing the identity necessary in the proof of
      Proposition~\ref{p:nhslide}.}}
    \label{f:slide-square}
\end{figure}
By the definition of sign assignments we have
\begin{equation}\label{eq:Square}
S(X)\cdot S(YZ) \cdot S(XY)\cdot S(Z)=-1.
\end{equation}
Now it is easy to see that (after consistently naming and orienting
the curves) the domains $D, D_1$ and $D_5$ are combinatorially
equivalent (i.e. the formal flows corresponding to them are equal).
In addition, the formal flow of $XY$ is the same as of $D_2$, $X$ and
$D_4$ differ in an $\alpha$-boundary degeneration (hence their
$S$-values are the same), and similarly $Z$ and $D_3$ differ by an
$\alpha$-boundary degeneration.  In a similar manner, $ZY$ and $D$
differ in an $\alpha$- and a $\beta$-boundary degeneration.  Therefore
the product in the left side of Equation~\eqref{eq:Square}
is equal to 
\[
-S(D_2)\cdot S(D_3)\cdot S(D_4)\cdot S(D_5),
\]
hence $S(D_2)\cdot S(D_3)\cdot S(D_4)\cdot S(D_5)=1$. Multiplying it
with $S(D)=S(D_1)$, the equality
\[
S(D)=S(D_1)\cdot S(D_2)\cdot S(D_3)\cdot S(D_4)\cdot S(D_5)
\]
follows at once. Notice that this is the identity required by the
argument of Lemma~\ref{l:alg} to establish that the map $\x \mapsto \x
+K$ from $\CFaa (\DD ; \Z )$ to $\CFaa (\DD '; \Z )/K$ induces an
isomorphism on homologies.

Suppose now that the domain $D$ connecting $\x$ and $\y$ is a
rectangle, and the nice handle slide happens along an arc contained by
one of the rectangles (necessarily on the boundary of $D$). As it was
discussed in the proof of \cite[Proposition~7.22]{nice}, we
distinguish various cases. Suppose that we slide $\alpha _1$ over the
curve $\alpha _2$. We have to examine the following cases: (a) the
rectangle is of width one, (b) the rectangle is of width at least two
and the side opposite to $\alpha _1$ is on a curve $\alpha _3$
distinct from $\alpha _2$ and finally (c) the side opposite to $\alpha
_1$ is on $\alpha _2$.

In case (a) above the domains before and after the handle slide are
shown in Figure~\ref{f:widthone}. The chain in $\DD '$ corresponding
to $D$ (in $\DD$) has been identified in
\cite[Proposition~7.22]{nice}. 
\begin{figure}
  \begin{center}
    \includegraphics{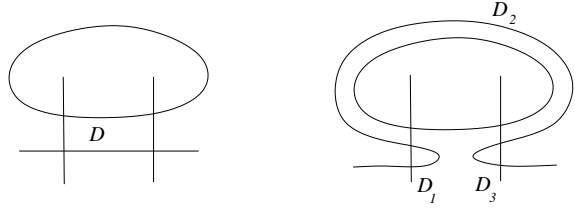}
  \end{center}
  \caption {{\bf The domains before and after the handle slide in 
case the rectangle is of width  one.}}
    \label{f:widthone}
\end{figure}
According to the result of Lemma~\ref{l:alg} we need to show that
\[
S(D)=-S(D_1)\cdot S(D_2)\cdot S(D_3).
\]
Consider now the square given by the diagram of Figure~\ref{f:square-wone}.
\begin{figure}
  \begin{center}
    \includegraphics{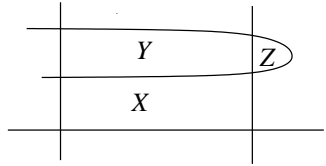}
  \end{center}
  \caption {{\bf The square used in the proof of
      Proposition~\ref{p:nhslide}.}}
    \label{f:square-wone}
\end{figure}
Then a simple observation shows that (after fixing appropriate labels
and orientations) $Z$ is the same as the domain $D_3$ after an
$\alpha$-boundary degeneration, $YZ$ is the same as $D_1$, $X$ agrees
with $D$ after a $\beta$-boundary degeneration, while $XY$ can be
identified with $D_2$ after an $\alpha$- and a $\beta$-boundary
degeneration. Hence the equality
\[
S(X)\cdot S(YZ)\cdot S(XY)\cdot S(Z)=-1
\]
given by the square transforms to 
\[
S(D)=-S(D_1)\cdot S(D_2)\cdot S(D_3),
\]
the equality we needed in accordance with Lemma~\ref{l:alg}.

Case (b) needs the application of more squares, hence we
provide a more detailed argument in this case. Suppose that 
the chain in $\DD '$ corresponding to $D$ (in $\DD$) is given as below:
  \begin{eqnarray*}
    \begin{diagram}
     \x= x_1 e_i {\mathbf t} & \qquad & f_i e_j {\mathbf t} & \qquad \qquad & f_k e_i {\mathbf t} & \qquad & \\
      & \rdTo^{D_1} & \dTo & \rdTo^{ D_3} & \dTo & \rdTo^{D_5} & \\
      &  & f_j e_i {\mathbf t} & & f_i e_k {\mathbf t} &  & y_1 e_i {\mathbf t}=\y \\
    \end{diagram}
  \end{eqnarray*}
(The rectangles given by the vertical arrows will be called $D_2$ and $D_4$.)
The schematic picture of this case is shown by Figure~\ref{f:hslide}.
\begin{figure}
  \begin{center}
    \includegraphics{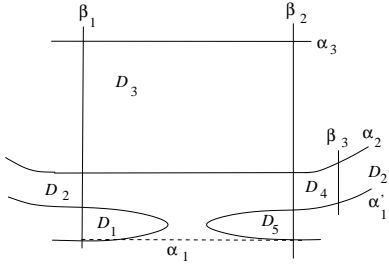}
  \end{center}
  \caption {{\bf Domains in a nice handle slide.}}
    \label{f:hslide}
\end{figure}

In the two diagrams $\DD$ and $\DD '$ the orientations of the curves
are fixed in a coherent manner (the orientation of $\alpha _1'$ is
induced from the orientation of $\alpha _1$). According to our
principle from Lemma~\ref{l:alg} (since the length of the chain is now
$n=2$), we need to show now that
\begin{align}\label{eq:otos}
S(D)=S(D_1)\cdot S(D_2)\cdot S(D_3)\cdot S(D_4)\cdot S(D_5).
\end{align}
\begin{figure}
  \begin{center}
    \includegraphics{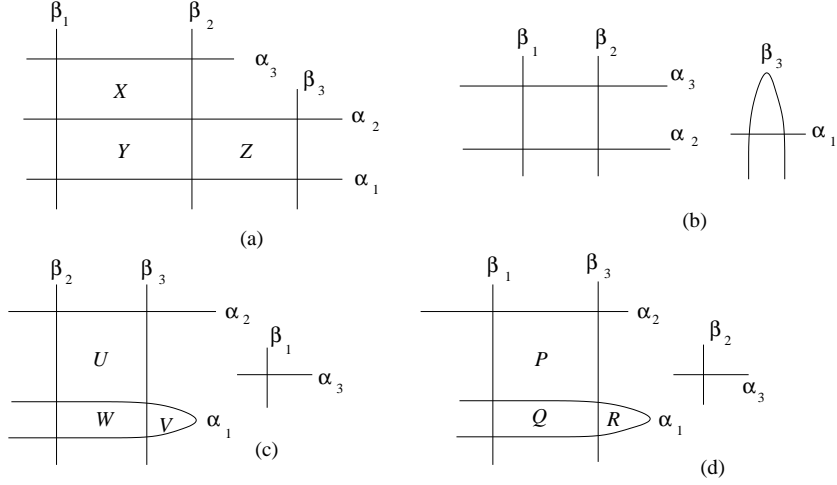}
  \end{center}
  \caption {{\bf The square of formal flows relevant in Case (b) of
      the proof of Proposition~\ref{p:nhslide}.}}
    \label{f:formal}
\end{figure}
Consider the four squares of formal flows given by
Figure~\ref{f:formal}, where the orientations are chosen according to
the chosen orientations of the corresponding curves in the Heegaard
diagram $\DD$. The formal flow corresponding to the domain $D$ of
$\DD$ is equal to $XY$, while the domain $D_1$ (in $\DD'$) is exactly
$QR$. The domain $D_4$ can be identified with $U$.  The domains $D_5$
and $VW$ differ by an $\alpha$-boundary degeneration (hence the sign
assignment $S$ takes the same values on them), and $D_2$ and $P$ also
differ by an $\alpha$-boundary degeneration. The domains $D_3$ and $X$
almost correspond to each other --- the only difference is that the
crossing of $\alpha _1$ and $\beta _3$ is oppositely oriented in the
two case.  The two possibilities appear in the relation associated to
Figure~\ref{f:formal}(b), where the two bigons in the square can be
identified with $V$ and $R$. (Recall that to be identical, one should
also check the signs of the intersections on the nonmoving
coordinates.)

Recall that the identity of
Property~(S-\ref{property:AntiCommutation}) corresponding to a square
can be conveniently rewritten as $\Pi _{i=1}^4 S(\phi _i)=-1$.
Therefore the four identities implied by the diagrams of
Figure~\ref{f:formal} are:

\begin{align*}
S(XY)\cdot S(Z)\cdot S(X)\cdot S(YZ)=-1\\
S(U)\cdot S(WV)\cdot S(UW)\cdot S(V)=-1\\
S(P)\cdot S(QR)\cdot S(PQ)\cdot S(R)=-1\\
S(X)\cdot S(V)\cdot S(R)\cdot S(D_3)=-1
\end{align*}
Furthermore, by noticing that $YZ$ and $PQ$ are combinatorially
identical (hence admit the same $S$-value), and similarly
$S(UW)=S(Z)$, we are ready to turn to the proof of
Equation~\eqref{eq:otos}:
\begin{align*}
S(D_1)\cdot S(D_2)\cdot S(D_3)\cdot S(D_4)\cdot S(D_5)\\
=S(QR)\cdot S(P)\cdot S(D_3)\cdot S(U)\cdot S(VW)\\
=(-1) S(PQ)\cdot S(R)\cdot S(D_3)\cdot S(U)\cdot S(VW)\\
=(-1) S(YZ)\cdot S(R)\cdot S(D_3)\cdot S(U)\cdot S(VW)\\
=(-1)^2 S(YZ)\cdot S(R)\cdot S(D_3)\cdot S(UW)\cdot S(V)\\
=(-1)^2 S(YZ)\cdot S(R)\cdot S(D_3)\cdot S(Z)\cdot S(V)\\
=(-1)^2 S(YZ)\cdot S(Z)\cdot S(D_3)\cdot S(R)\cdot S(V)\\
=(-1)^3S(YZ)\cdot S(Z)\cdot S(X)=S(XY)=S(D).
\end{align*}

A similar argument applies in the case when the side of the rectangle
$D$ opposite to $\alpha _1$ is on the curve $\alpha _2$ to which we
apply the handle slide (and the rectangle is of width more than 1). In
this case we need to distinguish two subcases, according to the
relative orientations of $\alpha _1$ and the opposite side. We leave
the details of this computation to the reader.
\end{proof}

After these preparations, we are ready to prove the invariance of the
homology groups under nice moves:

\begin{thm}\label{thm:invnicemove}
Suppose that $\DD$ is a nice diagram.  The homology group of the chain
complex $(\CFaa (\DD ; \Z ), \partialaa _{\DD }^{\Z})$ is (stably)
invariant under nice moves.
\end{thm}
\begin{proof}
Since a nice move is either a nice stabilization, a nice isotopy
or a nice handle slide, the proof of the statement follows from 
Propositions~\ref{p:nstab}, \ref{p:niso} and \ref{p:nhslide}. 
\end{proof}

\begin{proof}[Proof of Theorem~\ref{thm:main2}]
By composing the results of Theorems~\ref{thm:squarezero},
\ref{thm:indepp} and \ref{thm:invnicemove}, the result follows at
once.
\end{proof}

Suppose that $Y$ is a closed, oriented 3--manifold, and consider the
stable Heegaard Floer homology $\HFast (Y)$ of $Y$, as it is defined
in \cite[Definition~8.1]{nice}: Recall that in its definition we
consider a splitting of $Y$ as $Y_1\# _n S^1\times S^2$ (where $Y_1$
contains no $S^1\times S^2$--summand), fix a convenient diagram $\DD$
for $Y_1$ and consider the equivalence class of $H_* (\CFaa (\DD),
\partialaa _{\DD})$ (as the equivalence is given by
\cite[Definition~1.1]{nice}). This time, however, we consider the
chain complex over $\bfz $ and the boundary map also takes signs into
account.  To accomplish this, we need to fix an order on the
$\alphak$-- and $\betak$--curves of $\DD$ and also an orientation on
them.  In addition, we need to fix a sign assignment $S$ of power $n$
(where $n$ is the number of $\alphak$--curves). The resulting
equivalence class (of stable Heegaard Floer homology) will be denoted
by $\HFast (Y_1 ; \Z )$, and $\HFast (Y ; \Z )$ is given by taking its
tensor product with $(\Z \oplus \Z)^n$.  Now the combination of the
proof of \cite[Theorem~8.2]{nice} with the above argument of the
invariance of the homologies (with coefficients in $\Z$) under nice
moves readily implies

\begin{cor}\label{cor:indep}
The equivalence class $\HFast (Y; \Z )$ is a smooth invariant of the
oriented 3--manifold $Y$. \qed
\end{cor}

As in \cite[Section~9]{nice}, we can consider the theory will fully twisted
coefficients, providing the chain complex $(\CFa _T, \partiala _T)$. With the
aid of a sign assignment, once again, this chain complex can be considered
over $\Z$ rather than over $\Z/2\Z$ (as was discussed in \cite{nice}).
The invariance proofs of this section readily imply that
\begin{cor} 
The twisted Floer homology $\HFa _T(Y; \Z )$ of the 3-manifold $Y$ 
over the integers is a smooth invariant of $Y$. \qed
\end{cor}

\section{The existence and uniqueness of sign assignments}
\label{sec:fourth}
Now we turn to the proof of Theorem~\ref{thm:exunique}, the result
which played a crucial role in the arguments of the previous
section. Both the construction of a sign assignment, and the proof of
its uniqueness (up to gauge equivalence) will be first carried out on certain
subsets of formal flows, and then we patch the partial results
together. Notice first that the notions of sign assignments and their
gauge equivalences make sense on subsets.
\begin{defn}
  Let $Z$ be a set of formal generators and $E$ a set of formal flows
  connecting various of the formal generators in $Z$.  A {\bf sign
    assignment over $(Z,E)$} is a function $S\colon E \to \{\pm 1\}$
  satisfying the three properties of
  Definition~\ref{def:SignAssignment}. When we drop $E$ from this
  notation, then it is understood that $E$ denotes the set of all
  flows connecting any two formal generators in $Z$.
\end{defn}

We will distinguish certain subsets of the set of formal generators.

\begin{defn}
  Fix a permutation $\sigma$. Let $X(*,\sigma)$ denote the
  set of formal generators whose permutation agrees with $\sigma$
  (i.e. only the sign profile is allowed to vary). Similarly,
  if $\epsilon$ is some fixed sign profile, let $X(\epsilon,*)$
  denote the set of formal generators whose sign profile agrees with
  $\epsilon$ (i.e. the permutation is allowed to vary).
\end{defn}

\subsection{Orienting bigons}
\label{subsec:OrientBigons}
In the following we will examine sign assignments on the subsets
$X(*, \sigma )$ for some permutation $\sigma$. Notice that among such
generators we have only formal bigons (and any bigon connects two such
generators, for some choice of $\sigma$).

\begin{prop}
  \label{prop:DefineOverAllBigons}
For a fixed permutation $\sigma$ there is, up to gauge equivalence, a
unique sign assignment over the set of formal generators
$X(*,\sigma)$.
\end{prop}
\begin{proof}
  Consider first the case where $\sigma=e$ is the identity
  permutation.  We construct a sign assignment as follows. Suppose the
  bigon $\phi$ is supported in the $i^{th}$ factor. Define
\[
\sign(\phi)=\sign_0(\phi)\cdot \prod_{j=1}^{i-1}\epsilon_j,
\]
where $\epsilon =(\epsilon _1 , \ldots , \epsilon _n )$ is the sign
profile of $\phi $, and $\sign$ is one of the sign assignments we have
found in Lemma~\ref{lem:OrientBigons}.  (Here we think of $\phi$ as a
bigon of power 1, on the $i^{th}$ coordinate.)  It is easy to
verify that $\sign$ satisfies the required anticommutativity of
disjoint bigons.

Next we turn to the proof of uniqueness (up to gauge equivalence),
still assuming that $\sigma =e$. (In this case a formal generator is
specified by its sign profile $\epsilon $ only.)  Consider the graph
whose vertices are formal generators in $X(*,e)$, and whose edges are
the formal bigons.  Consider the following spanning tree $T$ of this
graph: take an edge connecting the two formal generators $\epsilon$
and $\epsilon '$ if these generators differ in exactly one position $i$,
and both assign $+1$ to all positions $j<i$. Represent this edge by
one of the four formal bigons (two if we fix the starting and the
terminal generator) connecting $\epsilon $ and $\epsilon '$.
Suppose now that $\sign , \sign '$ are two sign assignments given on
$X(*, e)$. Since $T$ is a tree, when restricting $\sign$ and $\sign '$
to $T$, these functions become gauge equivalent.  To show that the two
sign assignments are gauge equivalent over $X(*, e)$ as well, we show
that $\sign \vert _T$ (and similarly $\sign '\vert _T$) determines
$\sign$ (and $\sign '$, respectively).

First consider the graph $G$ we get from $T$ by adding those flows in
$X(*,e)$ which connect two formal generators connected by an edge in
$T$. By Lemma~\ref{lem:OrientBigons}, the extension of a sign
assignment from $T$ to $G$ is unique.  Next we extend the sign
assignment to those formal bigons which connect generators where the
signs before the moving coordinate $i$ are $+1$ with one single
exception (where the sign is therefore $-1$).  For each new formal
flow $f_1$ we can find three other flows $f_2$, $f_3$, and $f_4$ which
are in $G$, with the property that the pairs $(f_1, f_2) $ and $(f_3,
f_4)$ form a square. Thus, by
Property~(S-\ref{property:AntiCommutation}) in
Definition~\ref{def:SignAssignment}, the value $\sign (f_1)$ is
determined uniquely by $\sign (f_2)$, $\sign (f_3)$, and $\sign
(f_4)$. Let now $G_k$ denote those formal flows which connect formal
generators with the property that there are at most $k$ $(-1)$'s in
positions prior to the moving coordinate.  By the principle described
above, the sign assignment uniquely extends from $G_k$ to
$G_{k+1}$. Since $G_0=G$ and $G_n=X(*, e)$ (where we consider formal
flows and generators of power $n$), the uniqueness of the extension is
verified in this case.

Consider finally the case of an arbitrary permutation $\sigma$. If
$\phi$ is a bigon with moving coordinate in the $i^{th}$ coordinate,
connecting $(\epsilon,\sigma)$ with $(\epsilon',\sigma)$ (note that
$\epsilon_j=\epsilon'_j$ except when $i=j$), then we define
\[  
\sign(\phi)=\sign_0(\phi)\cdot \prod_{j=1}^{i-1}\epsilon_{\sigma(j)}.
\]
As before, the uniqueness up to gauge equivalence follows exactly as above.
\end{proof}
Later it will be important to notice that 
restricted gauge tramsformations act trivially on the restriction of
a sign assignment to any $X(*, \sigma )$. 

\subsection{Fixing the sign profile}
\label{subsec:FixSignProfile}

The aim of the present subsection is to prove the following:

\begin{prop}
  \label{prop:FixSignProfile}
  Fix the sign profile ${\bf 1}$ which is identically $1$ in each
  factor. There is a unique sign assignment up to gauge equivalence on
  the subset $X({\bf 1},*)$.
\end{prop}

By fixing the sign profile, we exclude all the bigons (since along a
bigon the sign of one of the crossings changes).  Sign convention for
rectangles in a similar context was worked out in \cite{MOST}, and in
the following we will rely on the results proved there.  (For a
further approach to constructing sign assignments on grid diagrams,
see \cite{Gal}.)  Specifically, we can view a permutation $\sigma$ as
a generator for the combinatorial Floer complex discussed in
\cite{MOST}.  Formal rectangles then correspond to actual rectangles
in the torus, and by appropriately orienting the grid diagram, the
sign profile for all generators will be ${\bf 1}$.  In~\cite{MOST} a
sign is associated to empty rectangles, i.e. to those which contain no
other point of the form $(i,\sigma(i))$ in their interiors. On the
other hand, we also need to assign signs to those formal rectangles
which give rise to non-empty rectangles in the chosen grid
representation.  Our first aim now is to define a sign assignment $S$
for possibly non-empty rectangles in the torus.

We will start our discussion by considering rectangles in the
\emph{planar} grid, that is, we cut the toroidal grid along an
$\alphak$- and along a $\betak$-curve $\alpha _0$ and $\beta_0$, and
examine only those rectangles of the toroidal grid which are disjoint
from these cuts. Let us define the {\em complexity} $K(r)$ of a
rectangle $r\colon \x \to \y$ to be the number of components $p$ of
$\x$ which are supported in the interior of $r$. In particular, an
empty rectangle has complexity equal to zero. For these rectangles the
result of Step 4 of \cite[Section~4]{MOST} shows the existence of an
appropriate sign assignment; indeed, \cite[Proposition~4.15]{MOST}
provides a formula for such a function $S$ on complexity zero
rectangles.

Suppose that $r$ has complexity greater than zero. Then there is a
component $p$ of $\x$ in the interior of $r$. The rectangle $r$ can be
viewed as a composite of three rectangles, two of which have $p$ as a
corner. Indeed, subdividing our rectangle into four regions (meeting
at $p$), $A$, $B$, $C$, and $D$, as indicated in
Figure~\ref{fig:DecomposeRectangle}, we can view the rectangle $r$ as
a composite of three rectangles in four different ways: $B*(AC)*D$,
$C*(BD)*A$, $B*(CD)*A$, or $C*(AB)*D$, cf.
Figure~\ref{fig:DecomposeRectangle} . We call the first of these a
{\em conventional decomposition}. Note that a conventional
decomposition depends on a choice of the point $p$ in the interior of
$r$.

\begin{figure}
    \begin{center}
      \input{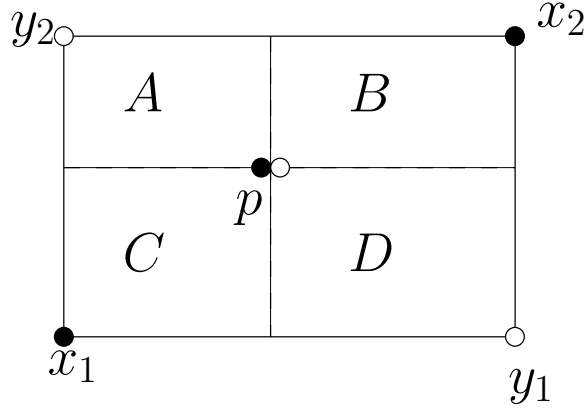}
    \end{center}
    \caption {{\bf Decompose a rectangle.}
      \label{fig:DecomposeRectangle}
      We have illustrated a rectangle from $p x_1 x_2 $ to $p y_1 y_2$
      with a component in its support (i.e. with complexity $\geq
      1$). This rectangle can be decomposed in four ways: $B*(AC)*D$,
      $C*(BD)*A$, $B*(CD)*A$, or $C*(AB)*D$. We will use the first
      decomposition (which we called the conventional decomposition).}
  \end{figure}

We now define $S$ inductively as follows:
\begin{enumerate}
\item if $r$ is an empty rectangle (i.e. one with $K(r)=0$),
  then $S(r)$ is the sign from~\cite{MOST}.
\item if $r$ is a rectangle with $K(r)>0$,
  and $B*(AC)*D$ is a conventional decomposition,
  then $S(r)$ is defined to be the product
  $S(B)\cdot S(AC)\cdot S(D)$ (where the three terms are
  defined because they have smaller complexity).
\end{enumerate}

\begin{rems}\label{rem:decompos}
\begin{itemize}
\item The definition above follows from the required property
of a sign assignment: denote the sides of the rectangles in 
$r$ as shown by Figure~\ref{f:motiv}(a), and consider the corresponding
square of flows given by Figure~\ref{f:motiv}(b).
\begin{figure}
  \begin{center}
    \input{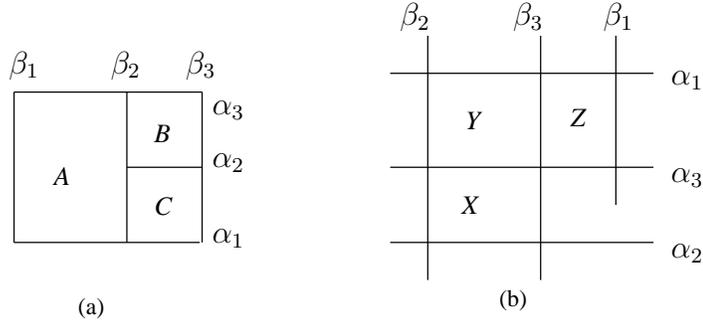}
  \end{center}
  \caption {{\bf The motivation for the extension
      rule.}}
    \label{f:motiv}
\end{figure}
It is not hard to see that, as formal flows, $X=B$. In addition, $Z$
and $ABC$ differ by one $\alpha$- and one $\beta$-boundary
degeneration, $XY$ and $C$ differ by a $\beta$-boundary degeneration,
and $YZ$ differs from $A$ by an $\alpha$- and a $\beta$-boundary
degeneration.  Since for a sign assignment $S(X)\cdot S(YZ)\cdot
S(Z)\cdot S(XY)=-1$, and the three $\beta$-boundary degenerations
introduce further negative signs (while the $\alpha$-degenerations do
not), we get
\[
S(ABC)\cdot S(B)\cdot S(A)\cdot S(C)=1,
\]
justifying our choice for $S(ABC)$.

\item The notation is a little inaccurate: the value of $S$ on a
  rectangle depends on the initial point of the underlying rectangle,
  not just its underlying region, so when we write an expression such
  as $S(B)\cdot S(AC)\cdot S(D)$, it should be understood that $AC$ is
  taken with initial point the terminal point of $B$: thus, the terms
  cannot be freely commuted. In order to keep notations manageable, we
  will keep the above (slightly sloppy) convention throughout the
  rest of the paper.

\item Notice that the conventional decomposition $B*(AC)*D$ differs
  from $B*(CD)*A$ by a square, and similarly, $C* (BD)*A$ and $C*
  (AB)*D$ differ by a square. Finally, the conventional decomposition
  differs from $C*(AB)*D$ by a square.  Our choice of the conventional
  decomposition is dictated by our initial choice of putting $1$ in
  (S-\ref{property:AlphaDegenerations}) and $-1$ in
  (S-\ref{property:BetaDegenerations}) of
  Definition~\ref{def:SignAssignment}.

\end{itemize}
\end{rems}

Since a conventional decomposition depends on a choice of a point $p$
in the interior of $r$, it would be more accurate to record all those
choices in the notation for $S$ as well.  According to the following
result, this is unnecessary:

\begin{lemma}
  \label{lem:WellDefinedForNonemptyRectangles}
  The above function $S$ satisfies the following properties:
  \begin{enumerate}
  \item
    \label{property:IndepOfDecomposition}
    If $r$ is a rectangle then its associated sign $S(r)$ is
    independent of the choice of the conventional decomposition.
   \item 
    \label{property:CommuteRectangles}
    If $r_1$ and $r_2$ are two rectangles and the pairs $(r_1, r_2)$,
    $(r_1',r_2')$ form a square, 
then $S(r_1)\cdot S(r_2)+S(r_1')\cdot S(r_2')=0$.
  \end{enumerate}
\end{lemma}

\begin{proof}
  We prove the statements simultaneously, by induction on the total
  complexity ($K(r)$ for the first statement, and $K(r_1)+K(r_2)=
  K(r_1')+K(r_2')$ for the second).

  To prove Property~\eqref{property:IndepOfDecomposition}, let $p_1$
  and $p_2$ be two components of $\x$ in the interior of $r$. There
  are two subcases, according to the relative positions of $p_1$ and
  $p_2$, as illustrated in Figure~\ref{fig:RelativePositions}.
  Specifically, the two points $p_1$ and $p_2$ give a subdivision of
  $r$ into nine rectangular regions. Denote the middle one by $E$. The
  points $\{p_1, p_2\}$ can be either the upper left and lower right
  corners of $E$ (as in the left-hand-side of
  Figure~\ref{fig:RelativePositions}), or they can be the upper right
  and lower left ones (as in the right-hand-side of
  Figure~\ref{fig:RelativePositions}).

  Consider the left-hand case. We can either first take a conventional
  decomposition at $p_1$, to get $r=(BC)*(ADG)*(EFHI)$, and then
  follow this by a conventional decomposition of $EFHI$ at $p_2$, to
  realize $r=(BC)*(ADG)*F*(EH)*I$. Alternatively, taking $p_2$ first
  and then $p_1$, we have a different decomposition
  $r=(CF)*(ABDEGH)*I=(CF*B*(ADG)*(EH)*I$.  But we have that
  \begin{align*}
    S_{p_1 p_2}(r) &=
    S(BC)\cdot S(ADG)\cdot S(F)\cdot S(EH)\cdot S(I) \\
    &=-S(BC)\cdot S(F) \cdot S(ADG) \cdot S(EH)\cdot S(I) \\
    &= S(CF)\cdot S(B)\cdot S(ADG)\cdot S(EH)\cdot S(I) \\
    &= S_{p_2 p_1}(r)
  \end{align*}
  where  we apply Property~\eqref{property:CommuteRectangles}
  twice (which is valid by the inductive hypothesis): first to the
  square $(ADG,F, F, ADG)$, and then to the square $(BC,F, CF,B)$.

  Similarly, in the second case, we have
  \begin{align*}
    S_{p_1 p_2}(r) &=
    S(C)\cdot S(BE)\cdot S(F)\cdot S(ADG)\cdot S(HI) \\
    &= 
    -S(C)\cdot S(BE)\cdot S(ADG) \cdot S(F)\cdot S(HI) \\
    &= 
    S(C)\cdot S(BE)\cdot S(ADG) \cdot S(H)\cdot S(FI) \\
    &=S_{p_2 p_1}(r),
  \end{align*}
  where we have used Property~\eqref{property:CommuteRectangles}
  twice again: For the squares $(ADG,F, F, ADG)$ and $(F,HI,H,FI)$.  This
  completes the verification of
  Property~\eqref{property:IndepOfDecomposition}.
  \begin{figure}
    \begin{center}
      \input{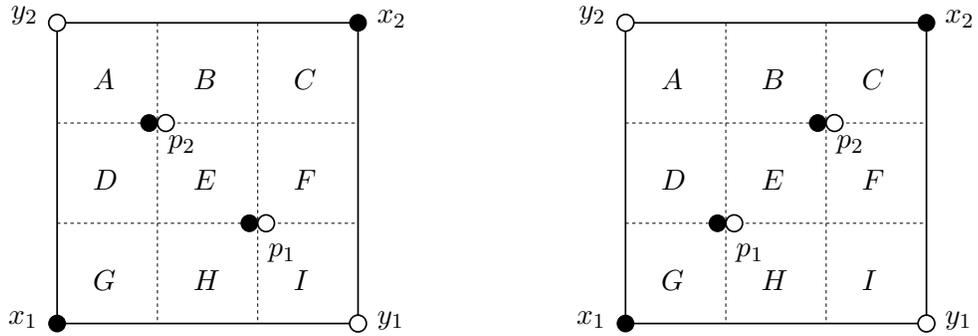}
    \end{center}
    \caption {{\bf Independence of conventional decomposition.}
      \label{fig:RelativePositions}
      Let $p_1$ and $p_2$ be two different components of $\x$ in the
      interior of a rectangle $r$ from $\x$ to $\y$. These two
      different points give a decomposition of $r$ into nine regions.
      Moreover, they give two different conventional decompositions of
      $r$. The combinatorics can be subdivided according to the
      relative positions of $p_1$ and $p_2$, as pictured here.}
  \end{figure}

  The proof of Property~\eqref{property:CommuteRectangles} can be
  subdivided into two subcases: in case (a) the rectangles $r_1$ and
  $r_2$ share a moving coordinate, while in case (b) the moving
  coordinates are disjoint. 

The verification of the equality in case (a) requires an examination of 
twelve subcases. Namely, the two rectangles can be positioned relative to each 
other in the planar grid in four possible ways, shown by the four $L$-shaped
domains of Figure~\ref{f:pain}. 
\begin{figure}
  \begin{center}
    \includegraphics{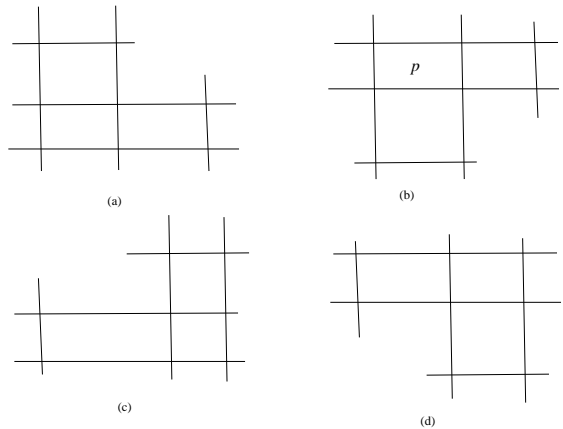}
  \end{center}
  \caption {{\bf The four main cases.} By putting $p$ in one of the three
domains, each case gives rise to three subcases. We give the details of the
argument for the configuration shown by (b).}
    \label{f:pain}
\end{figure}
For complexity zero domains the result of \cite{MOST} provides the
equality, hence we can assume that the complexity $K(r_1)+K(r_2)$ is
positive. Now each subcase gives rise to three further subcases,
depending on where the further coordinate in the three possible
domains is located. We will provide the argument in one case, leaving
the straightforward adaptation of the proof of the remaining cases to
the reader. So assume that $(r_1, r_2)$ is positioned as in
Figure~\ref{f:pain}(b), and one of the points (called $p$) showing
$K(r_1)+K(r_2)>0$ is located in the domain marked with a $p$.  We will
use induction on the joint complexity, and therefore (as instructed by
the definition of $S$) we subdivide the domains of the configuration
as it is shown by Figure~\ref{f:diamonds}(a). The square corresponding
to this configuration is shown by Figure~\ref{f:diamonds}(b),
\begin{figure}
  \begin{center}
    \includegraphics{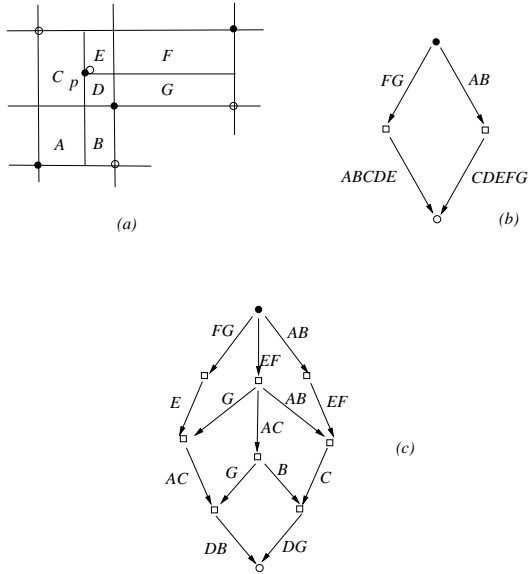}
  \end{center}
  \caption {{\bf The proof of anticommutativity.} In the diagram
    arrows indicate the connecting flows, the full circle stands for
    the starting while a hollow circle for the terminal formal
    generator.  The flows are decomposed as compositions of further
    formal flows; the intermediate formal generators are all denoted
    by hollow squares.}
    \label{f:diamonds}
\end{figure}
and we need to show
that 
\[
S(AB)\cdot S(CDEFG)\cdot S(FG)\cdot S(ABCDE)=-1.
\]
(Once again, throughout the proof we will be sloppy by specifying the
flows only with the letters of the underlying domains, although the
further intersections and their signs are equally important. These
further data can be easily derived from the diagram.)  Now
Figure~\ref{f:diamonds}(c) shows a partition of the square into five
sub-squares, and for all of these the inductive hypothesis shows that
the corresponding product is equal to $-1$. Since there are five such
sub-squares, the product of their contribution is also equal to $-1$.
The sides of the octagon give the sides of the square of
Figure~\ref{f:diamonds}(b) after expanding them by the definition of
$S$ on rectangles of positive complexity, completing the argument for
this particular subcase. The proof of the further eleven subcases
follow the same line of reasoning, giving the decomposition of the
square in question into an odd number of sub-squares for which the
inductive hypothesis applies and therefore conludes the proof.

Case (b) --- where the moving coordinates of $r_1$ and $r_2$ are disjoint ---
can be handled as follows. We distinguish for subcases:
\begin{enumerate}
\item the two rectangles do not contain each other's corners,
\item the two rectangles contain one of each other's corners,
\item one rectangle contains two of the corners of the other rectangle, and
finally
\item one rectangle contains the other one.
\end{enumerate}
A similar argument as before expands the square under consideration
and decomposes it into an odd number of smaller squares for which
induction holds. The desired relation for the original square then
easily follows. Instead of giving the detailed arguments in each case
above, we provide the schematic diagrams from which the proofs can be
easily recovered. Indeed, Figure~\ref{fig:esetek1} shows the idea for proving
\begin{figure}
  \begin{center}
    \includegraphics[width=9cm]{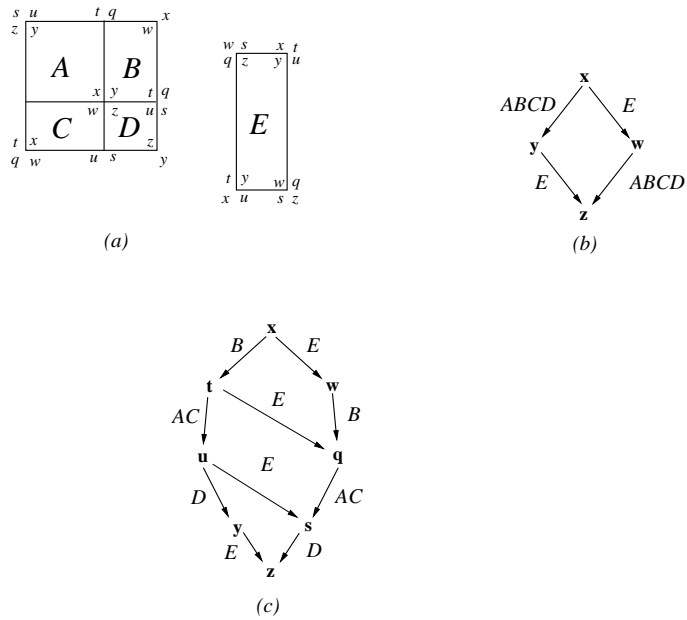}
  \end{center}
  \caption{{\bf The proof of the square when the two rectangles do not
      contain each other's corners.} In the diagram we show the further
    specialization when, in fact, the rectangles are disjoint. If the
    interiors of the rectangles intersect, but the corners are not in
    each other, the same scheme applies.}
    \label{fig:esetek1}
\end{figure}
the first subcase above, Figure~\ref{fig:esetek2} shows 
how to handle the second,
\begin{figure}
  \begin{center}
    \includegraphics{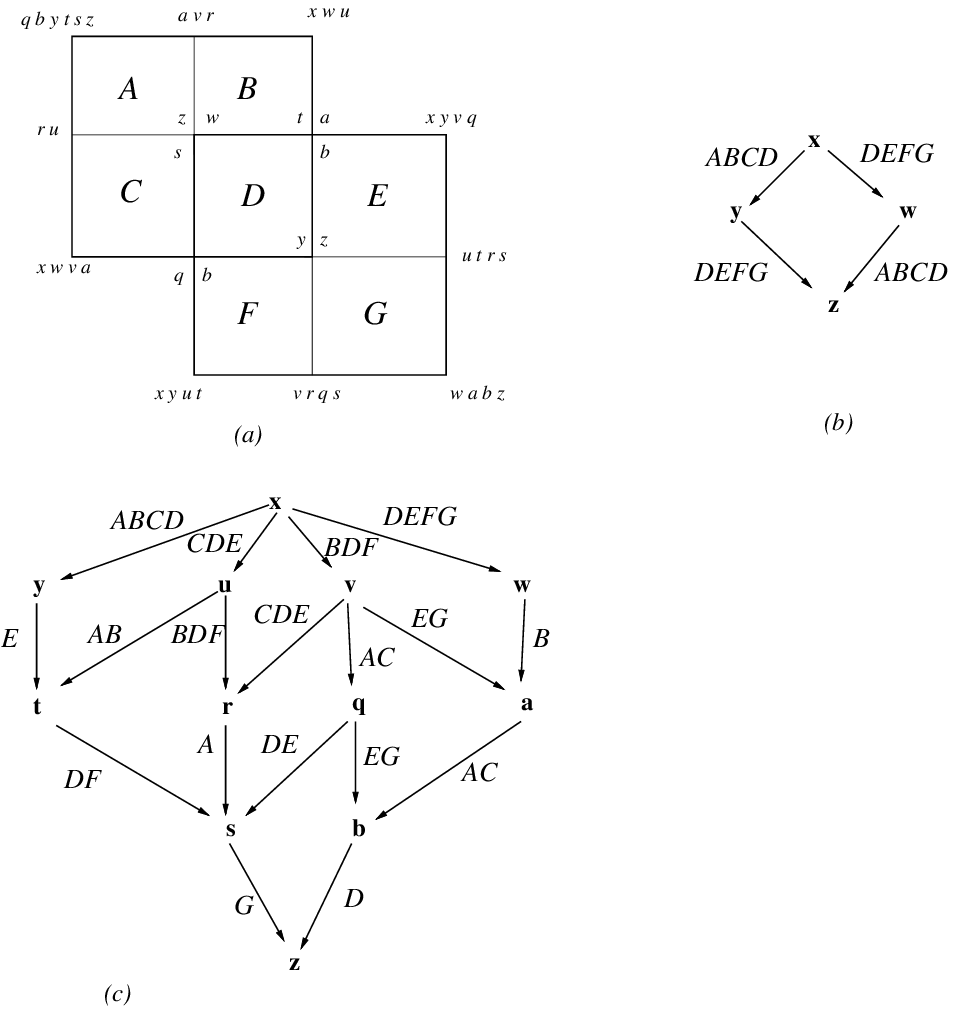}
  \end{center}
  \caption{{\bf The proof of the square when the two rectangles 
      contain one of each other's corners.}}
    \label{fig:esetek2}
\end{figure}
Figure~\ref{fig:esetek3} deals with the case when one rectangle contains
two of the other's corners,
\begin{figure}
  \begin{center}
    \includegraphics[width=9cm]{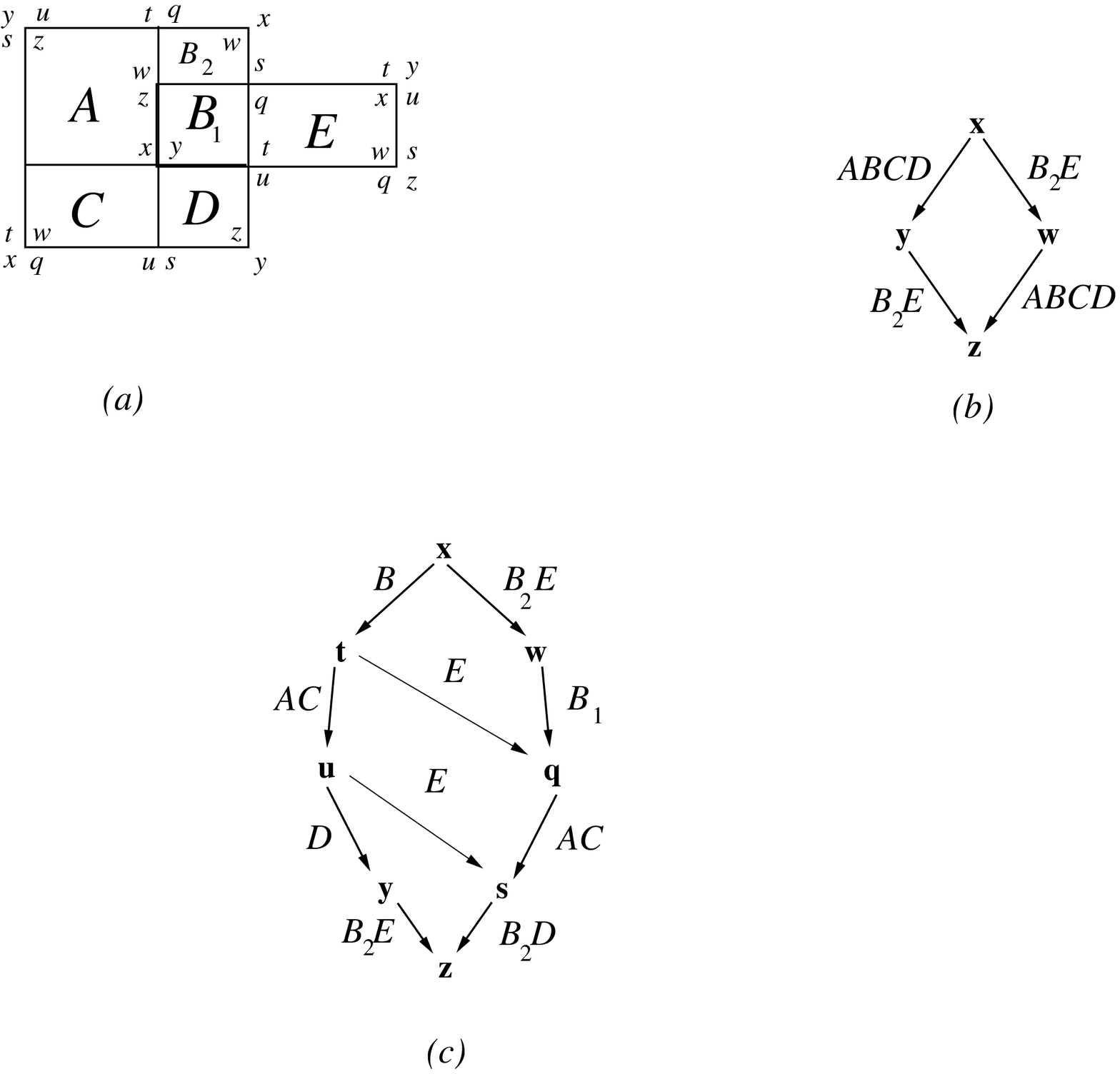}
  \end{center}
  \caption{{\bf The proof of the square when one of the two rectangles 
      contains two  corners of the  other's.}}
    \label{fig:esetek3}
\end{figure}
and finally Figure~\ref{fig:esetek4} shows the case when one rectangle is
contained by the other.
\begin{figure}
  \begin{center}
    \input{esetek4.pstex_t}
  \end{center}
  \caption{{\bf The proof of the square when one of the two rectangles 
      contains the other one.}}
    \label{fig:esetek4}
\end{figure}
In all of the above cases induction completes the arguments and
concludes the proof of the lemma.
\end{proof}

Now we are in the position to define the value of the sign assignment
for any rectangle on the toroidal grid.

\begin{defn}
  Suppose that $G$ is a given toroidal grid, with two circles $\alpha
  _0$ and $\beta _0$ specified, along which we cut it into a planar
  grid.  Suppose that $r$ is a given rectangle on the toroidal
  grid. If $r$ is disjoint from the curves $\alpha _0$ and $\beta _0$,
  then it gives rise to a planar grid and the value of $S$ has been
  defined for it by the previous discussion. If $r$ is disjoint from
  $\beta _0$ but intersects $\alpha _0$, then an application of a
  $\beta$-boundary degeneration provides a rectangle $r'$ for which
  $S$ is already defined (as it is in the planar grid) and its
  $S$-value is related to $S(r)$ by the formula $S(r)\cdot
  S(r')=-1$. This specifies $S(r)$. A similar argument gives the value
  of $S(r)$ in terms of an $\alpha$- (and a combination of an
  $\alpha$- and a $\beta$-)boundary degeneration in the further
  remaining cases.
\end{defn}

In order to complete the discussion, we need to verify that the
definition above provides a sign assignment.

\begin{lemma}
  \label{lem:GeneralAnticommutation}
  If the two pairs $(r_1, r_2)$ and $(r_1', r_2')$ in $X({\bf 1},*)$
  form a square, then $S(r_1)\cdot S(r_2)+S(r_1')\cdot S(r_2')=0$
\end{lemma}
\begin{proof}
We begin with some terminology. If the rectangles $r$ and $s$ form an
$\alpha$-boundary resp. $\beta$-boundary degeneration, then we call
$s$ the $\alpha$-degenerate resp. $\beta$-degenerate \emph{companion}
to $r$. Moreover, if $s$ is the $\alpha$-degenerate companion to $r$,
and $t$ is the $\beta$-degenerate companion to $s$, we call $t$ the
$\alpha$-$\beta$-companion to $r$.

Suppose that $(r_1, r_2, r_1', r_2')$ is a given square in $X({\bf 1},
*)$.  If both $r_1$ and $r_2$ (and therefore $r_1'$ and $r_2'$) are
planar, i.e.  disjoint from $\alpha _0, \beta _0$, then
Lemma~\ref{lem:WellDefinedForNonemptyRectangles} implies the result.
If the moving coordinates of $r_1$ and $r_2$ are disjoint, then by
taking the appropriate companions of those rectangles which intersect
$\alpha _0$ (or $\beta _0$, or both), we can reduce the problem to the
planar case.

Suppose next that $r_1$ and $r_2$ share a moving coordinate. In this
case $r_1* r_2$ contains two segments $d_1, d_2$ along which we get
the two different decompositions (as $r_1*r_2$ and as $r_1'* r_2'$).
We will label them so that $d_1$ is horizontal and $d_2$ is vertical.
If $\alpha _0, \beta _0$ are disjoint from $d_1, d_2$, then the
previous argument applies.

Suppose that $d_2$ intersects $\alpha _0$, but $d_1$ is disjoint from
$\beta_0$.  In this case only one of the four rectangles $(r_1, r_2,
r_1', r_2')$ is planar.  Suppose that the planar rectangle is $r_2$.
To simplify matters, assume that $\beta _0$ is disjoint from $r_1,
r_2$. Let $s_1$, $s_1'$, and $s_2'$ be the $\beta$-degenerate
companions for $r_1$, $r_1'$, and $r_2'$ respectively. In this case,
$s_1$ is a rectangle, which decomposes as $s_1=r_2*s_2'*s_1'$.  This
decomposition differs by two squares from the conventional
decomposition, and hence $S(s_1)=S(r_2)\cdot S(s_2')\cdot
S(s_1')$. Since this equation involves three $\beta$-degenerations, it
can be rewritten as the desired relation $S(r_1) S(r_2)=-S(r_1')\cdot
S(r_2')$.  The other subcase (where $r_1$ is the planar rectangle)
follows similarly. The case where $d_2$ is disjoint from $\alpha_0$,
but $d_1$ intersects $\beta_0$ follows similarly as well.

In the case $d_2$ intersects $\alpha_0$ and $d_1$ intersects
$\beta_0$, we argue as follows. First, observe that either both $r_1$
and $r_1'$ meet $\alpha_0$ and $\beta_0$, or both $r_2$ and $r_2'$
meet $\alpha_0$ and $\beta_0$. Consider the first subcase (i.e.  $r_1$
and $r_1'$ meet $\alpha_0$ and $\beta_0$). Now, $r_2$ and $r_2'$ each
meet exactly one of $\alpha_0$ and $\beta_0$. By renumbering, we can
assume that $r_2$ meets $\beta_0$ and $r_2'$ meets $\alpha_0$. Let
$t_1$ and $t_1'$ be the $\alpha$-$\beta$-degenerate companions to
$r_1$ and $r_1'$; and let $t_2$ be the $\beta$-degenerate companion to
$r_2$ and $t_2'$ be the $\alpha$-degenerate companion to
$r_2'$. Observe that $t_1$, $t_2$, $t_1'$, and $t_2'$ are planar. Now
we can find rectangles $u_1$ and $u_2$ with the property that
$(t_1,u_1)$ and $(t_1',u_2)$ form a square; as does
$(t_2',u_1)$ and $(t_2,u_2)$. We conclude that
$$S(r_1)S(r_2) S(r_2')S(r_1')=-S(t_1)S(t_2)S(t_2')S(t_1')=-1.$$
The subcase where both $r_2$ and $r_2'$ meet both $\alpha_0$ and
$\beta_0$ follows similarly.
\end{proof}
 
\begin{proof}[Proof of Proposition~\ref{prop:FixSignProfile}]
  Recall that by \cite{MOST} the sign assignment exists and is unique
  up to gauge equivalence on the rectangles giving rise to empty
  rectangles in the planar grid. Now the extension from empty
  rectangles to arbitrary (still in the planar grid) and from planar
  to toroidal was uniquely determined by the axioms of a sign
  assignment, and our previous results verified the existence.
  Indeed, by our definition the properties regarding boundary
  degenerations come for free, while
  Property~(S-\ref{property:AntiCommutation}) of
  Definition~\ref{def:SignAssignment} about a square is exactly the
  content of Lemma~\ref{lem:GeneralAnticommutation}.
\end{proof}

\subsection{Varying permutations and sign profiles}
\label{subsec:vary}
After having the sign assignment for fixed permutations (involving
only bigons) and fixed sign profiles (allowing only rectangles), now
we consider subsets where we allow the variation of permutations and
sign profiles as well.  

\begin{defn}
  Let $r\colon \x\to \y$ be a formal rectangle.  For any non-moving
  coordinate of $r$ (i.e. a point $p \in \x\cap \y$), consider the new
  formal rectangle $r'\colon \x'\to \y'$ which is obtained as follows:
  $\x'$ (and $\y '$) is gotten from $\x$ (and $\y$, resp.) by
  switching the value of the sign profile at $p\in\x\cap\y$. In this
  case, we say that $r$ and $r'$ are related by a {\bf simple flip}.
  If $r$ and $r'$ can be connected by a sequence of rectangles
  $r=r_1,r_2,\dots,r_{m+1}=r'$, with the property that $r_i$ and
  $r_{i+1}$ differs by a simple flip for all $i=1, \ldots ,m$ then we
  say that $r$ and $r'$ determine the same {\bf type} of rectangle.
Let $\theta (r) $ denote the set of rectangles having the same type
as $r$.
\end{defn}  
Note that if $r$ and $r'$ are related by a simple flip, then we can
  find some pair of bigons $b$ and $b'$ with the property that the
  pairs $(b, r)$ and $(r',b')$ form a square.
 
\begin{lemma}
  \label{lem:VaryWithinType}
  Let $S$ be a sign assignment defined over
  all bigons, and over some fixed rectangle $r$
  connecting two generators with the same sign profile ${\bf 1}$.
  This sign assignment can be uniquely extended to
  all rectangles $r'$ which have the same type as $r$. 
\end{lemma}

\begin{proof}
  We define the {\em sign complexity} of a generator $\x$ to be the
  number of places where the underlying sign profile is $-1$.  For a
  rectangle $R$, its sign complexity is defined to be the sign
  complexity of its initial generator.  If $R$ is a rectangle with
  positive sign complexity $m$, then there is a bigon $B$ with the
  property that the two pairs $(R,B)$ and $(B',R')$ form a square, $R$
  and $R'$ are rectangles of the same type, $B$ and $B'$ are bigons,
  and the sign complexity of $R'$ is one less than the sign complexity
  of $R$.

  We can now inductively define $S(R)$ to satisfy $S(R) =-S(B)\cdot
  S(B')\cdot S(R')$.  This definition does not lead to a
  contradiction: Suppose that the rectangle $R_1$ can be gotten in two
  different ways from rectangles of sign complexity one less. Then
  there is a single rectangle $R_2$ with sign complexity two less,
  with the property that
  $$A*B*R_1=R_2*A*B,$$
  where here $A$ and $B$ are both disjoint bigons.
  Thus, $S(R_1)$ is determined either by
  $$S(A)\cdot S(B)\cdot S(R_1)=S(R_2)\cdot S(A)\cdot S(B)$$
  or by
  $$S(B)\cdot S(A)\cdot S(R_1)=S(R_2)\cdot  S(B)\cdot S(A);$$
  but by Property~(S-\ref{property:AntiCommutation}) for bigons 
these equations are equivalent.
  
Thus, these relations uniquely determine $S(R)$ for any rectangle
$R$ of the same type as $r$. By construction, the extension of $S$
satisfies Property~(S-\ref{property:AntiCommutation}).  It is easy to
see that Properties~(S-\ref{property:AlphaDegenerations}) and
(S-\ref{property:BetaDegenerations}) are preserved, as well: Suppose
that $Q$ and $R$ are rectangles forming a pair of boundary
degeneration. This, in particular, means that they have the same
moving coordinates.  By choosing an appropriate pair $B,B'$ of bigons
we can reduce the sign complexity of $(Q,R)$:
  \begin{align*}
    S(B)\cdot S(Q)\cdot S(R)&=-S(Q')\cdot S(B')\cdot S(R) \\
    &=S(Q')\cdot S(R')\cdot S(B);
  \end{align*}
  Then the equality  $S(Q)\cdot S(R)=S(Q')\cdot S(R')$ and induction 
on the sign complexity of $(Q,R)$ implies the result.
\end{proof}

\begin{defn}
Fix a rectangle $r$ and consider the $16$ different
rectangles gotten by changing orientations of the edges of $r$. Denote the set
of rectangles obtained in this manner by $\omega(r)$. 
\end{defn}
The relevance of this definition is given by the following simple fact:
\begin{lemma}\label{l:kislemma}
  For any formal rectangle $r$ there is a formal rectangle $r_1$ such
  that the sign profile of $r_1$ is ${\bf 1}$ and $\omega (r_1)$
  contains a rectangle $r_2$ which has the same type as $r$.
\end{lemma}
\begin{proof}
  Obviously, by possibly reversing the orientations on the edges of
  $r$ and reversing the orientation of one of the arcs at each
  non-moving coordinate where the sign profile is $-1$, we get a new
  formal rectangle $r'$ which has the desired sign profile ${\bf 1}$.
  The claim then easily follows.
\end{proof}

Next we will extend the sign assignment to $\omega (r)$ once the value
is fixed on bigons and on $r$.  Let us fix a rectangle in $\omega
(r)$.  For each of the four edges of this rectangle, and each endpoint
$v$ of each of these edges, we can consider the relation gotten by
juxtaposing a rectangle and a bigon based at $v$. We call these the
{\em basic relations}. This gives, in all, $16$ relations between the
sign assignment associated to the various (pairs of) rectangles in
$\omega(r)$.  Two rectangles $r_1$ and $r_2$ in $\omega(r)$ can be
connected by one of the basic relations if $r_2$ is gotten by
reversing the orientation of one of the edges of $r_1$.
If $r_1$ and $r_2$ are connected by a basic relation, they are in fact
connected by $4$ basic relaltions (see
Figure~\ref{fig:FourBasicRelations}). We show that all four of these
relations coincide.

\begin{figure}
\begin{center}
\input{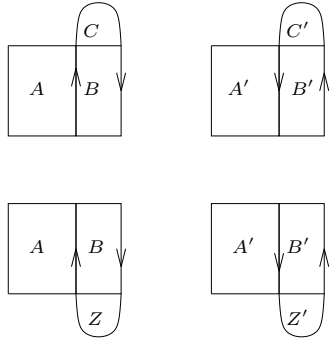}
\end{center}
\caption {{\bf Four basic relations connecting the same two rectangles.}
  After orienting the three remaining boundary arcs in these four figures
  (in the same manner), we obtain four different relations connecting
  the same two rectangles $r_1=A=A'B'$ and $r_2=AB=A'$.
\label{fig:FourBasicRelations}}
\end{figure}

\begin{lemma}
  \label{lemma:OneBasicRelation}
  If $r_1$ and $r_2$ are connected by a basic relation, then all four
  basic relations connecting them are equivalent.
\end{lemma}

\begin{proof}
To this end, observe that in Figure~\ref{fig:FourBasicRelations}
we have the identity
$S(A)=S(A'B')$ (as these rectangles are combinatorially
indistinguishable); and similarly $S(AB)=S(A')$.
Thus, if we write $r_1$ for $A$ and $r_2$ for $AB$, the four pictures
give the following relations between $S(r_1)$ and $S(r_2)$:
\begin{align*}
  S(A)\cdot S(BC)&=-S(C)\cdot S(AB) \\
  S(A')\cdot S(B'C')&=-S(C')\cdot S(A'B') \\
  S(AB)\cdot S(Z)&=-S(BZ)\cdot S(A) \\
  S(A'B')\cdot S(Z')&=-S(B'Z')\cdot S(A'). \\
\end{align*}
We claim that these four relations are all equivalent. We start by
showing the equivalence of the first two.  Note first that $C$ and
$C'$ differ in the orientation of one of their sides, and that is
either an $\alphak$ or a $\betak$-side. This distinction provides two
subcases. In the first case, according to Lemma~\ref{lem:OrientBigons}
(see especially Remark~\ref{rmk:BigonOrientations}), $S(C)=-S(C')$ and
$S(BC)=-S(B'C')$, while in the second case $S(C)=S(C')$ and
$S(BC)=S(B'C')$. In either case, the first two relations are evidently
the same. The equivalence of the last two follows similarly.

Next, we show the equivalence of the first and third. Juxtaposing the
two pictures, we note that the first equation is equivalent to 
\begin{equation}
  \label{eq:FirstRelation}
  \pm S(A)= S(A)\cdot S(BC)\cdot S(Z) = - S(C)\cdot S(AB)\cdot S(Z)
\end{equation}
where the sign in the first term is $+1$ if $BC*Z$ is an $\alpha$-boundary
degeneration, and $-1$ if it is a $\beta$-boundary degeneration.
Similarly, the second equation is equivalent to:
\[
\pm  S(A)= S(C) \cdot S(BZ)\cdot  S(A) = -S(C)\cdot S(AB)\cdot S(Z)
\]
which is the same as the conclusion from
Equation~\eqref{eq:FirstRelation}. This identity finishes the proof of
the lemma.
\end{proof}

\begin{lemma}
  \label{lem:ChangeOrientationsOfSides}
  A sign assignment $S$ which is defined over all bigons and on a
  fixed rectangle $r$ can be uniquely extended to a function on all
  the rectangles in $\omega(r)$ in such a way that the extension 
  satisfies Property~(S-\ref{property:AntiCommutation}) whenever
  $\phi_1$ and $\phi_2$ are pairs, one of which is a rectangle, and
  the other is a contiguous bigon.
\end{lemma}

\begin{proof}
  Clearly, any two rectangles in $\omega(r)$ can be connected by
  a sequence of basic relations. Thus, the value of $S(r)$ determines
  $S(r')$ for any $r'\in \omega(r)$. We must verify that there are no
  contradictions.

  To this end, suppose that $S(r_1)$ and $S(r_2)$ are connected by an
  elementary relation, and $S(r_2)$ and $S(r_3)$ are also connected by
  an elementary relation, and $r_3\neq r_1$. These combine to give a
  relation ${\mathcal {R}}$ between $S(r_1)$ and $S(r_3)$ (by
  eliminating $S(r_2)$).  There is another orientation $r_2'$, so that
  $S(r_1)$ and $S(r_2')$ are connected by an elementary relation, as
  are $S(r_2')$ and $S(r_3)$.  These combine to give another relation
  ${\mathcal {R}}'$ between $S(r_1)$ and $S(r_3)$.  We claim that
  ${\mathcal {R}}$ and ${\mathcal {R}}'$ are equivalent; the lemma
  then follows from this observation.  To verify the claim, consider
  Figure~\ref{fig:CommutingRelations}.  This illustrates the case
  where $r_1$ and $r_3$ differ in the orientations of two consecutive
  sides.

\begin{figure}
\begin{center}
\input{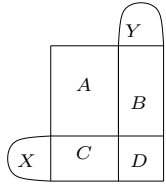}
\end{center}
\caption {{\bf Commuting basic relations.} 
\label{fig:CommutingRelations}}
\end{figure}
 
  Write $r_1=A$, $r_2=AC$, $r_3=ABCD$. Then we have $r_2'=AB$.
  The basic relations between $r_1$, $r_2$ and $r_3$ are:
  \begin{align*}
    S(A)\cdot S(BY)&=-S(Y)\cdot S(AB) \\
    S(AB)\cdot S(XCD)&= -S(X)\cdot S(ABCD),
  \end{align*}
  which combine to give the relation ${\mathcal {R}}$:
  \begin{equation}
    \label{eq:R}
    S(A)\cdot S(BY)\cdot S(XCD)=S(Y)\cdot S(X)\cdot S(ABCD);
  \end{equation}
  while the basic relations between $r_1$, $r_2'$ and $r_3$ are:
  \begin{align*}
    S(X)\cdot S(AC)&=-S(A)\cdot S(XC) \\
    S(AC)\cdot S(YBD)&= -S(Y)\cdot S(ABCD),
  \end{align*}
  which combine to give the relation ${\mathcal {R}}'$:
  \begin{equation}
    S(A)\cdot S(XC)\cdot S(YBD)=S(X)\cdot S(Y)\cdot S(ABCD).
  \end{equation}
  (Note again that the bigons $X$ and $Y$ appearing in relation
  ${\mathcal {R}}'$ differ from the corresponding bigons appearing in
  ${\mathcal {R}}$; they have the same support, but they connect
  different generators.)  Now, the relations ${\mathcal {R}}$ and
  ${\mathcal {R}}'$ are equivalent, since $S(X)\cdot S(Y)=-S(Y)\cdot
  S(X)$ and $S(BY)\cdot S(XCD)=-S(XC)\cdot S(YBD)$, by properties of
  the sign assignment for bigons.

  There is a second case to consider, where $r_1$ and $r_3$ differ in
  the orientations of two opposite sides. We leave this case to the
  interested reader.
\end{proof}
Summarizing the previous results, we have

\begin{lemma}
  \label{lem:ExtendToEverything}
  Let $S$ be a sign assignment defined over all bigons and over some
  fixed rectangle $r$ connecting two fixed generators.  Then $S$ can
  be uniquely extended to a function over all rectangles in
  $\cup \{ \omega(r_1) \mid r_1\in \theta (r)\}$ such that the extension satisfies
  Property~(S-\ref{property:AntiCommutation}).
\end{lemma}

\begin{proof}
  We extend the sign assignment to $\theta (r)$ as in
  Lemma~\ref{lem:VaryWithinType}, and extend further to the elements
  of $\omega (r_1)$ (with $r_1\in \theta (r)$) by
  Lemma~\ref{lem:ChangeOrientationsOfSides}. These two extensions are
  compatible, according to Property~(S-\ref{property:AntiCommutation})
  for bigons. By both constructions,
  Property~(S-\ref{property:AntiCommutation}) still holds for any two
  formal flowlines in the set.  
\end{proof}

\subsection{The definition of a sign assignment}
\label{ssec:def}

Lemma~\ref{lem:ExtendToEverything}, together with
Lemma~\ref{l:kislemma} and the constructions from
Subsections~\ref{subsec:OrientBigons} and \ref{subsec:FixSignProfile}
now allows us to consistently define the function $S$ over any formal
flow: start with the sign assignment $S$ given over all rectangles
connecting generators with sign profile ${\mathbf 1}$
(Proposition~\ref{prop:FixSignProfile}), and define it also over all
bigons as in Proposition~\ref{prop:DefineOverAllBigons}.  Together,
these two pieces of data allow us to define $S$ also for all the
remaining formal flows. By the previous subsection, this extension is
well-defined. It remains to verify that the extension $S$ still
satisfies all the properties of a sign assignment.

\begin{lemma}
  \label{lem:CommutationRelPreserved}
  The extension $S$ satisfies Property~(S-\ref{property:AntiCommutation})
  for all pairs of formal flows.
\end{lemma}

\begin{proof}
  If $\phi_1$ and $\phi_2$ are both bigons, this follows from
  Proposition~\ref{prop:DefineOverAllBigons}. If $\phi_1$ and $\phi_2$
  are chosen so that one of them is a rectangle and the other is a
  disjoint bigon, then this follows from
  Lemma~\ref{lem:VaryWithinType}. If the bigon is not disjoint, this
  was verified in Lemma~\ref{lem:ChangeOrientationsOfSides}.

  Suppose next that $\phi_1$ and $\phi_2$ are both
  rectangles whose four sides are oriented in a standard manner.
  Then, we verify Property~(S-\ref{property:AntiCommutation}) by
  induction on the sign complexity of the initial generator, with the
  base case given by Proposition~\ref{prop:FixSignProfile}.  Represent
  $\phi_1$ by $A$ and $\phi_2$ by $BC$, $\phi_3$ by $C$, and $\phi_4$
  by $AB$,  and let $X$ be a disjoint bigon. Suppose that the
  inductive hypothesis gives $S(A)\cdot S(BC)=-S(C)\cdot S(AB)$, and
  that the sign complexity of $A'$, $BC'$, $AB'$, and $C'$ (gotten by
  switching the sign in the factor where $X$ is supported) is one
  greater than the sign complexity of the corresponding rectangles
  $A$, $BC$, $AB$, and $C$.  Then, applying
  Property~(S-\ref{property:AntiCommutation}) in the case of a
  rectangle and a disjoint bigon (twice), we see that:
  \begin{align*}
    S(A)\cdot S(BC)\cdot S(X) &=
    S(A)\cdot S(X')\cdot S(BC') \\
    &=
    S(X'')\cdot S(A')\cdot S(BC');
  \end{align*}
  and similarly
  $S(C)\cdot S(AB)\cdot S(X)=S(X'')\cdot S(A')\cdot S(BC')$.
  The inductive step now follows easily.
  
  Having verified Property~(S-\ref{property:AntiCommutation}) for
  rectangles whose sides have standard orientation, it remains to
  see that the defining property remains true as the orientations of
  the sides are reversed.  There are two subcases: either the reversed
  side is shared by $\phi_1$ and $\phi_2$, or it is not,
  see Figure~\ref{fig:Commu}.

  \begin{figure}
    \begin{center}
      \input{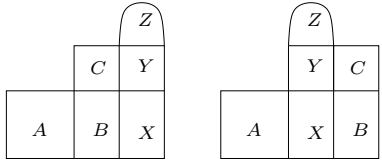}
    \end{center}
    \caption {{\bf Proof of Lemma~\ref{lem:CommutationRelPreserved}.}
      Preservation of Property~(S-\ref{property:AntiCommutation}) under
      orientation reversal of sides. The two subcases are illustrated
      here.}
\label{fig:Commu}
  \end{figure}

  First we turn to the case where the reversed edge is not shared;
  this appears on the left in
  Figure~\ref{fig:Commu}.  In the notation from that
  figure, our aim is to show that if $S(C)\cdot S(AB)=S(A)\cdot
  S(BC)$, then $S(CY)\cdot S(ABX)=S(A)\cdot S(BCXY)$.  This follows
  from the facts that:
  \begin{align*}
    S(C)\cdot S(AB)\cdot S(XYZ) &=
    -S(C)\cdot S(YZ)\cdot S(ABX) \\
    &= S(Z)\cdot S(CY)\cdot S(ABX)
  \end{align*}
  (by two applications of Property~(S-\ref{property:AntiCommutation})
  for a rectangle and a bigon)
  and
  \begin{align*}
    S(A)\cdot S(BC)\cdot S(XYZ) &= -S(A)\cdot S(Z)\cdot S(BCXY) \\
    &= S(Z)\cdot S(A)\cdot S(BCXY)
  \end{align*}
  (by two applications of Property~(S-\ref{property:AntiCommutation});
  one for a rectangle and a bigon, and another for a pair of disjoint bigons).
  These two equations, together with the hypothesis that
  $S(C)\cdot S(AB)=S(A)\cdot S(BC)$, give
  $S(CY)\cdot S(ABX)=S(A)\cdot S(BCXY)$.

  Finally, in the case where the reversed edge is shared, we use
  notation from the right on Figure~\ref{fig:Commu}.
  We wish to show that the condition that
  $S(A)\cdot S(XYBC)=-S(YC)\cdot S(AXB)$ is equivalent to
  $S(AX)\cdot S(BC)=-S(C)\cdot S(AXB)$.
  This follows from the fact that
  \begin{align*}
    S(A)\cdot S(XYBC)\cdot S(Z) &= -S(A)\cdot S(XYZ)\cdot S(BC) \\
    &= S(YZ)\cdot S(AX)\cdot S(BC) .
  \end{align*}
  \end{proof}

\begin{lemma}
  \label{lem:BoundaryDegRelPreserved}
  If $r_1$ and $r_2$ are two rectangles so that $(r_1,r_2)$ is an
  $\alpha$- or $\beta$-boundary degeneration, then $S(r_1')\cdot
  S(r_2')=\pm 1$, where $r_1'\in \omega(r_1)$ and $r_2'\in
  \omega(r_2)$ are oriented compatibly so that $(r_1',r_2')$ is a
  boundary degeneration.  Here, of course,
  $$\pm 1 = \left\{
    \begin{array}{ll}
      +1 & {\text{if $(r_1',r_2')$ is an $\alpha$-boundary degeneration.}} \\
      -1 & {\text{if $(r_1',r_2')$ is a $\beta$-boundary degeneration.}} \\
    \end{array}
  \right.$$
\end{lemma}

\begin{proof}
  First, we show that if $r_1$ and $r_2$ intersect along some pair of
  edges, and $r_1'\in\omega(r_1)$ and $r_2'\in\omega(r_2)$ are gotten
  from $r_1$ and $r_2$ by reversing the orientation of one of
  the edges along which $r_1$ and $r_2$ meet, then
  \begin{equation}
    \label{eq:PreserveBoundaryDegenerationRelation}
    S(r_1)\cdot S(r_2)=S(r_1')\cdot S(r_2') .
  \end{equation}
  
  \begin{figure}
    \begin{center}
      \input{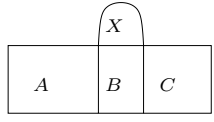}
    \end{center}
    \caption {{\bf Proof of Lemma~\ref{lem:BoundaryDegRelPreserved}.}
      \label{fig:BoundaryDegRelPreserved}
      Preservation of the boundary degeneration relation under
      orientation reversal of sides.}
  \end{figure}

  Following the conventions from
  Figure~\ref{fig:BoundaryDegRelPreserved}, we can write $r_1=AB$,
  $r_2=C$, and $r_1'=A$, $r_2'=BC$.  Now, 
  \begin{align*}
    S(X)\cdot S(AB)\cdot S(C)&=
    -S(A)\cdot S(BX) \cdot S(C) \\
    &= S(A)\cdot S(BC)\cdot S(X)\\
    &= \pm S(X)\\
    &= S(X)\cdot S(BC)\cdot S(A),
  \end{align*}
  verifying Equation~\eqref{eq:PreserveBoundaryDegenerationRelation}
  in the case where we reverse the orientation along one of the edges
  where $r_1$ and $r_2$ meet.

  We turn our attention now to
  Equation~\eqref{eq:PreserveBoundaryDegenerationRelation} in the case
  where we reverse the orientation along one of the other edges of
  $r_1$ and $r_2$. Suppose, for definiteness, that the rightmost
  edges of $r_1$ and $r_2$ are reversed in $r_1'$ and $r_2'$,
  while $r_1$ and $r_2$ meet along their two horizontal edges, as in
  Figure~\ref{fig:OtherDegRelPreserved}. We write $r_1=A$ and $r_2=X$,
  so that $r_1'=AB$ and $r_2'=XY$.

  \begin{figure}
    \begin{center}
      \input{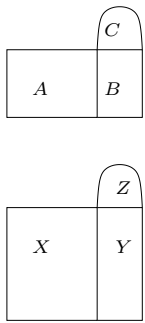}
    \end{center}
    \caption {{\bf Another part of the proof of Lemma~\ref{lem:BoundaryDegRelPreserved}.}
      \label{fig:OtherDegRelPreserved}
      $A$ and $X$ are complementary rectangles, and so are $B$ and
      $Y$.}
  \end{figure}

  Now, we know that
  \begin{align*}
    S(C)\cdot S(AB)&= S(A)\cdot S(BC) \\
    S(Z)\cdot S(XY)&= S(X)\cdot S(YZ).
  \end{align*}
  On the other hand, notice that $BC$ and $Z$ represent, formally, the same
  bigon, as do $C$ and $YZ$. Thus, we conclude that
  $$S(AB)\cdot S(XY)=S(A)\cdot S(B),$$ as desired.
\end{proof}

\begin{proof}[Proof of Theorem~\ref{thm:exunique}
(and hence of Theorem~\ref{thm:main1})]
Define the sign assignment $S$ on ${\mathcal {F}}_n$ by choosing a
sign assignment on the bigons (as it is given by
Proposition~\ref{prop:DefineOverAllBigons}), and independently on
formal rectangles connecting formal generators with constant sign
assignment {\bf {1}} (as it is described by
Proposition~\ref{prop:FixSignProfile}).  Use
Lemma~\ref{lem:ExtendToEverything} repeatedly for every rectangle
with constant sign assignment {\bf1}
to extend this partially defined
function to $S\colon {\mathcal {F}}_n \to \{ \pm 1\}$. By
Lemmas~\ref{lem:CommutationRelPreserved} and
\ref{lem:BoundaryDegRelPreserved} this extension will be, indeed, a
sign assignment. This argument then verifies the existence part of the
theorem.

Suppose now that $S$ and $S'$ are two sign assignments on ${\mathcal
  {F}}_n$. According to Proposition~\ref{prop:DefineOverAllBigons} the
two functions are gauge equivalent on the bigons.  Let $u\colon
{\mathcal {G}}_n \to \{ \pm 1 \}$ be such a gauge equivalence.
According to Proposition~\ref{prop:FixSignProfile}, when restricted to
the set of rectangles connecting formal generators with sign profile
constant {\bf {1}}, the two maps $S$ and $S'$ are gauge equivalent (on
this set of formal generators). Consider such a gauge equivalence and
let $u'$ denote its unique extension to ${\mathcal {G}}_n$ as a
restricted gauge equivalence. Now the gauge transformation $v=u\cdot
u'\colon {\mathcal {G}}_n \to \{ \pm 1\}$ has the property that $S^v$
and $S'$ are identical on bigons and on rectangles connecting formal
generators of constant sign profile {\bf {1}}. By the uniqueness of
the extension results of Subsections~\ref{subsec:vary} and
\ref{ssec:def}, this identity implies that $S^v=S'$ on ${\mathcal
  {F}}_n$, concluding the proof of the uniqueness part of the theorem.
\end{proof}

\end{document}